\documentclass[10pt]{amsart}
\usepackage[margin=1.2in]{geometry}
\usepackage[notcite,notref]{showkeys}
\usepackage{color}
\newtheorem{theorem}{Theorem}[section]
\newtheorem{lemma}[theorem]{Lemma}
\newtheorem{proposition}{Proposition}[section]
\newtheorem{corollary}{Corollary}[section]
\theoremstyle{definition}
\newtheorem{definition}[theorem]{Definition}

\theoremstyle{remark}
\newtheorem{remark}[theorem]{Remark}

\numberwithin{equation}{section}

\usepackage{amsmath,amsfonts,amssymb,amsthm}
\usepackage{mathtools}
\usepackage{commath}

\begin{document}
	\title{Stationary BGK models for chemically reacting gas in a slab}
\author{Doheon Kim}
\address{School of Mathematics, Korea Institutefor Advanced Study, Seoul 02455, Korea (Republic of)}
\email{doheonkim@kias.re.kr}
\author{Myeong-Su Lee }
\address{Department of Mathematics, Sungkyunkwan University, Suwon 440-746, Korea (Republic of)}
\email{cmsl3573@skku.edu}
\author{SEOK-BAE YUN }
\address{Department of Mathematics, Sungkyunkwan University, Suwon 440-746, Korea (Republic of)}
\email{sbyun01@skku.edu}
	\begin{abstract}
		We study the boundary value problem of two stationary BGK-type models - 
the  BGK model for fast chemical reaction and the BGK model for slow chemical reaction - 
and provide a unified argument to establish the existence and uniqueness of stationary flows of reactive BGK models in a slab.
For both models, the main diffculty arises in the uniform control of the reactive parameters from above
and below, since, unlike the BGK models for non-reactive gases, the reactive parameters for the reactive BGK models are  defined through highly nonlinear relations. 
To overcome this difficulty, we introduce several nonlinear functionals that capture essential
structures of such nonlinear relations such as the monotonicity in specific variables,
that enable one to derive necessary estimates for the reactive equilibrium coefficients.
   \end{abstract}
	\maketitle
	\section{introduction}
	The classical BGK model \cite{BGK} describes the relaxation process of the Boltzmann equation in a simpler setting.
	Due to its reliable performances in reproducing qualitative features of the Boltzmann equation in a numerically amenable way, the BGK model has
	been popularly used in place of the Boltzmann equation in many fields of rarefied gas dynamics. 
	As a model equation of the Boltzmann equation, the BGK model also inherits various modeling assumptions of the Boltzmann equation: The gas molecules are assumed to be
	non-ionized, monatomic, elastic and non-reactive. Efforts to remove any of these assumptions usally involve
	much complications and difficulties. And for each such removal of the assumptions, relevant BGK models were proposed. Regarding the removal of the non-reactiveness assumption,
	which is practically very important since the chemical reaction of gases shows up in various physical situations such as combustion processes, hypersonic flows around space vehicles, many efforts have been made in a series of works.
	
	The first relaxation type model for the system of reacting gases was suggested by Monaco and Pandolfi in \cite{MB}, providing the relaxational approximation of the reactive Boltzmann equation of Rossani and Spiga \cite{RS}. 
	The consistent BGK model for mixture problem derived in \cite{AAP} was extended  in \cite{GS} for gas system undergoing slow chemical reactions, existence of which is studied in the current work. 
	Brull derived a reactive BGK model for which the relaxation operator is split into the elasitic part and the chemical reaction part  in \cite{MB}. Extension to polyatomic reacting gas can be found in \cite{BisiS}, and the relaxational model for irreversible reactive chemical transformation is considered in \cite{BCMR}. For the study of shock problems of reactive gases using BGK type models, see \cite{GST}.
	
	To the best knowledge of authors, the existence issue of any of such reactive BGK models has never been
	considered in the literature, which is the main motivation of the current work. In this paper, we study the stationary problems for reactive BGK models involving the bimolecular fast reaction or slow chemical reactions. More precisely, we consider the stationary problems in a slab for reactive BGK models proposed in \cite{GS} (slow reaction) and \cite{GRS} (fast reaction). The term ``slow reaction" and ``fast reaction" are coined to compare the time scale of the chemical reactions measured up againt the time scale of elastic collisions. Slow reaction denotes the case in which the chemical reaction occurs over a time scale longer than the elatic collisions, and the fast reaction denotes the opposite case. \newline



	The paper is organized as follows: In Section 2, we introduce the reactive BGK models we are studying in this paper. In Section 3, we present our main result. Section 4 is devoted to estimate the macroscopic parameters. In Section 5, we define our solution spaces and formulate our problem as a fixed point problem in the solution space.
In Section 6, we show that our solution maps are invariant in the solution space under the assumption of Theorem \ref{main} and Theorem \ref{main2}.
In Section 7, we finish the proof by establishing Cauchy estimates for the solution maps.
 
	\section{BGK models for chemically reacting gases}
	In this section, we introduce two reactive BGK models we are considering in this work, for slow and fast chemical reaction respectively.
	We first define various coefficients and quantities shared by both models, and set up notational conventions:\newline
	
	\noindent$\bullet$ {\bf Velocity distribution function:} In the following, the velocity distribution function $f_i(x,v)$ $(i=1,2,3,4)$  represents the number density of $i$th molecule at the position $x\in [0,1]$ with velocity $v=(v_1,v_2,v_3)\in\mathbb{R}^3$. \newline
	
	\noindent$\bullet$ {\bf Physical and chemical constants:} 
	
	\noindent$(a)$  $\tau$ is the Knudsen number defined by the ratio of the mean free path and the characteristic length of the system. It measures how rarefied the system is.\\
	$(b)$ Mass: $m_i$ represents the mass for each species and $M$ denotes the total mass involved in the reaction process: $M=m_1+m_2=m_3+m_4$. And we use $\mu_{ij}:=\frac{m_im_j}{m_i+m_j}$ $(i,j=1,2,3,4)$ for the reduced mass.\\
	$(c)$ Energy: $E_i$ denotes the energy of the chemical bond and $\Delta E=-\sum_{i=1}^{4}\lambda_iE_i$ is the energy threshold, where we denote $\lambda_1=\lambda_2=-\lambda_3=-\lambda_4=1.$\\
	$(d)$ Interactions: $\chi_{ij}$ denotes the interaction coefficient, and the microscopic collision frequencies are denoted by $\nu_{ij}$. 
	$\chi_{ij}$ and $\nu_{ij}$ must satisfy the relation: $\chi_{ij}\leq\nu_{ij}$ (See \cite{AAP}) (We mention that this condition is necessary in the proof of Lemma \ref{lb}). We use $\nu_{12}^{34}$ and $\nu^{12}_{34}$ to denote the chemical microscopic collision frequency for the first and second model is respectively. See (\ref{nu slow}).\newline
	
	\noindent$\bullet$ {\bf Macroscopic fields:} We now define the macroscopic fields to construct the reactive equilibrium for the first model.\newline
	
	\noindent$(a)$ Single component macroscopic fields:
	\begin{align}\label{single}
	\begin{split}
	\rho^{(i)}&:= m_i n^{(i)}:=m_i\int_{\mathbb{R}^3}f_idv,\\
	\rho^{(i)}U^{(i)}&:=m_i\int_{\mathbb{R}^3}v  f_i dv,\\
	3k\rho^{(i)} T^{(i)}&:= m_i^2\int_{\mathbb{R}^3}|v-U^{(i)}|^2 f_idv.\cr
	\end{split}
	\end{align}
		
	\noindent$(b)$ Global macroscopic fields: 
	\begin{align}\label{global}
	\begin{split}
	&n=\sum_{i=1}^{4}n^{(i)},\qquad \rho=\sum_{i=1}^{4}\rho^{(i)},\qquad U=\frac{1}{\rho}\sum_{i=1}^{4}\rho^{(i)}U^{(i)},\\
	&nkT=\sum_{i=1}^{4}n^{(i)}kT^{(i)}+\frac{1}{3}\sum_{i=1}^{4}\rho^{(i)}(|U^{(i)}|^2-|U|^2).
	\end{split}
	\end{align}
	
	Now, we are ready to derive the parameters determining reactive Maxwellian $\mathcal{M}_i$, and present our models.
	\subsection{BGK model for slow chemical reaction}
	Our first model is proposed in \cite{GS} and describes the dynamics for Maxwellian molecules with slow chemical reactions. The stationary problem in a slab of which reads
	\begin{align}\label{system}
	v_1\frac{\partial f_i}{\partial x}=\frac{\nu_i}{\tau}\left(\mathcal{M}_i -f_i\right)\ \text{on}\ [0,1]\times\mathbb{R}^3,\quad (i=1,2,3,4)
	\end{align}
	subject to the boundary data:
	\[
	f_i(0,v)=f_{i,L}(v), \text{on}\ v_1>0,\quad
	f_i(1,v)=f_{i,R}(v),\text{on}\  v_1<0,\]
	with the reactive mawellians $\mathcal{M}_i$ defined by
	\begin{align*}
	\mathcal{M}_i=n_i\bigg(\frac{m_i}{2\pi k T_i}\bigg)^{3/2} \exp{\bigg(-\frac{m|v-U_i|^2}{2 kT_i}\bigg)},
	\end{align*}
	where the  reactive parameters are determined  as follows:
	First, we define the  collision frequencies $\nu_i$ by
	\begin{align}\label{nu fast}
	\begin{split}
	\nu_1=&\sum_{j=1}^4 \nu_{1j}n^{(j)}+\frac{2}{\sqrt{\pi}}\Gamma\bigg(\frac{3}{2},\frac{\Delta E}{kT}\bigg)\nu_{12}^{34}n^{(2)},\cr
	\nu_2=&\sum_{j=1}^4 \nu_{2j}n^{(j)}+\frac{2}{\sqrt{\pi}}\Gamma\bigg(\frac{3}{2},\frac{\Delta E}{kT}\bigg)\nu_{12}^{34}n^{(1)},\cr
	\nu_3=&\sum_{j=1}^4 \nu_{3j}n^{(j)}+\frac{2}{\sqrt{\pi}}\Gamma\bigg(\frac{3}{2},\frac{\Delta E}{kT}\bigg)\bigg(\frac{\mu^{12}}{\mu^{34}}\bigg)^{3/2}e^{\Delta E/kT}\nu_{12}^{34}n^{(4)},\cr
	\nu_4=&\sum_{j=1}^4 \nu_{4j}n^{(j)}+\frac{2}{\sqrt{\pi}}\Gamma\bigg(\frac{3}{2},\frac{\Delta E}{kT}\bigg)\bigg(\frac{\mu^{12}}{\mu^{34}}\bigg)^{3/2}e^{\Delta E/kT}\nu_{12}^{34}n^{(3)}.\cr
	\end{split}
	\end{align}
	Then we define  $n_i$, $U_i$ and $T_i$ as
	\begin{align}\label{B-1}
\begin{aligned}
	n_i&=n^{(i)}+\frac{\lambda_i}{\nu_i}\mathcal{S},\\
	m_in_iU_i&=m_in^{(i)}U^{(i)}+\frac{2}{\nu_i}\sum_{j=1}^{4}\chi_{ij}\mu_{ij}n^{(i)}n^{(j)}(U^{(j)}-U^{(i)})+\frac{\lambda_i}{\nu_i}m_iU\mathcal{S},\\
	\frac{3}{2}n_ikT_i&=\frac{3}{2}n^{(i)}kT^{(i)}-\frac{1}{2}m_i[n_i|U_i|^2-n^{(i)}|U^{(i)}|^2]+\frac{6k}{\nu_i}\sum_{j=1}^{4}\chi_{ij}\frac{\mu_{ij}}{m_i+m_j}n^{(i)}n^{(j)}(T^{(j)}-T^{(i)})\\
	&\quad +\frac{2}{\nu_i}\sum_{j=1}^{4}\chi_{ij}\frac{\mu_{ij}}{m_i+m_j}n^{(i)}n^{(j)}(m_iU^{(i)}+m_jU^{(j)})(U^{(j)}-U^{(i)})\\
	&\quad +\frac{\lambda_i}{\nu_i}\mathcal{S}\bigg[\frac{1}{2}m_i|U|^2+\frac{3}{2}kT+\frac{M-m_i}{M}kT\frac{(\Delta E/kT)^{3/2} e^{-\Delta E/kT}}{\Gamma(\frac{3}{2},\frac{\Delta E}{kT})}-\frac{1-\lambda_i}{2}\frac{M-m_i}{M}\Delta E\bigg],
\end{aligned}
	\end{align}
	where the quantity  $\mathcal{S}$ is defined by
	\begin{align*}
	\mathcal{S}=\nu_{12}^{34}\frac{2}{\sqrt{\pi}}\Gamma\bigg(\frac{3}{2},\frac{\Delta E}{kT}\bigg)\bigg[n^{(3)}n^{(4)}\bigg(\frac{m_1m_2}{m_1m_2}\bigg)^{3/2}e^{\Delta E/kT}-n^{(1)}n^{(2)}\bigg],
	\end{align*}

	 We note that the collision frequencies are given by the combination of non-ractive part, which takes the same form with the original BGK model for non-reacting gases \cite{AAP, BGK}, and the chemcial reaction part.
	\subsection{BGK model for fast chemical reaction}
	Our second model is proposed in \cite{GRS}, and it represents gas mixtures with fast chemical reactions:
	\begin{align}\label{system2}
	v_1\frac{\partial f_i}{\partial x}=\frac{\tilde\nu_i}{\tau}\big(\widetilde{\mathcal{M}}_i -f_i\big)\ \text{on}\ [0,1]\times\mathbb{R}^3,\quad (i=1,2,3,4)
	\end{align}
	subject to boundary data:
	\[ f_i(0,v)=f_{i,L}(v)\ \text{on}\ v_1>0,\quad f_i(1,v)=f_{i,R}(v)\ \text{on}\  v_1<0.\]
	The reactive maxwellian $\widetilde{\mathcal{M}}_i$ is defined by
	\begin{align*}
	\widetilde{\mathcal{M}}_i:=\tilde n^{i}\bigg(\frac{m_i}{2\pi k \tilde T}\bigg)^{3/2} \exp{\bigg(-\frac{m_i|v-\tilde U|^2}{2 k\tilde T}\bigg)}.
	\end{align*}
	The reactive parameters for this model are determined implicitly through the following procedure.\newline
		\noindent(a) First, we define the collision frequencies $\tilde{\nu}_i$  as follows:
	\begin{align}\label{nu slow}
	\begin{split}
	\tilde{\nu}_1=&\sum_{j=1}^4 \nu_{1j}n^{(j)}+\bigg(\frac{\mu^{34}}{\mu^{12}}\bigg)^{3/2}e^{-\Delta E/kT}\nu_{34}^{12}n^{(2)},\\
	\tilde{\nu}_2=&\sum_{j=1}^4 \nu_{2j}n^{(j)}+\bigg(\frac{\mu^{34}}{\mu^{12}}\bigg)^{3/2}e^{-\Delta E/kT}\nu_{34}^{12}n^{(1)},\\
	\tilde{\nu}_3=&\sum_{j=1}^4 \nu_{3j}n^{(j)}+ \nu_{34}^{12}n^{(4)},\\
	\tilde{\nu}_4=&\sum_{j=1}^4 \nu_{4j}n^{(j)}+ \nu_{34}^{12} n^{(3)}.
	\end{split}
	\end{align}
	\noindent(b) We define a function $F(x)$ by
	\begin{align}\label{B-30}
	\begin{aligned}
	F(x)&:=\frac{\bigg\{\sum_{i=1}^4\tilde{\nu}_i n^{(i)}\Big[\frac{1}{2}m_i(|U^{(i)}|^2-|\tilde U|^2)+\frac{3}{2}kT^{(i)}\Big]+\Delta E\tilde\nu_1(x-n^{1})\bigg\}}{\bigg(\frac{3}{2}k\sum_{i=1}^4 \tilde\nu_i n^{(i)}\bigg)}.
	\end{aligned}
	\end{align}
	With this definition of $F$, we first  define $\tilde n_{1}$ as the unique root of the equation
	\begin{equation}\label{B-3}
	\frac{\tilde\nu_3\tilde\nu_4}{\tilde\nu_1\tilde\nu_2}\frac{\tilde\nu_1x[ \tilde\nu_2n^{(2)}+\tilde\nu_1(x- n^{(1)})]}{[\tilde\nu_3n^{(3)}-\tilde\nu_1(x-  n^{(1)})][ \tilde\nu_4 n^{(4)}-\tilde\nu_1( x- n^{(1)})]}\exp\left(-\frac{\Delta E}{kF(x)}\right)=\left(\frac{\mu^{12}}{\mu^{34}}\right)^{3/2}
	\end{equation}
	or, equivalently,
	\begin{equation}\label{B-33}
	\frac{\tilde\nu_3\tilde\nu_4}{\tilde\nu_1\tilde\nu_2}\frac{\tilde\nu_1\tilde n_1[ \tilde\nu_2n^{(2)}+\tilde\nu_1(\tilde n_{1}- n^{(1)})]}{[\tilde\nu_3n^{(3)}-\tilde\nu_1(\tilde n_{1}-  n^{(1)})][ \tilde\nu_4 n^{(4)}-\tilde\nu_1( \tilde n_{1}- n^{(1)})]}\exp\left(-\frac{\Delta E}{kF(\tilde n_{1})}\right)=\left(\frac{\mu^{12}}{\mu^{34}}\right)^{3/2}
	\end{equation}
	in the domain defined by the constraint of positivity for density and temperature fields, i.e.,
	\begin{align*}
	\begin{aligned}
	&\tilde n_{1}>0, \quad\tilde n_{1}>n^{(1)}-\frac{\tilde\nu_2}{\tilde\nu_1}n^{(2)}, \quad  \tilde n_{1}<n^{(1)}+\frac{\tilde\nu_3}{\tilde\nu_1}n^{(3)},\quad \tilde n_{1}<n^{(1)}+\frac{\tilde\nu_4}{\tilde\nu_1}n^{(4)}, \\
	&\hspace{1cm}\tilde n_{1}>n^{(1)}-\frac{1}{\tilde\nu_1}\frac{1}{\Delta E} \sum_{i=1}^4\tilde\nu_i n^{(i)}\Big[\frac{1}{2}m_i(|U^{(i)} -\tilde U|^2)+\frac{3}{2}kT^{(i)}\Big].
	\end{aligned}
	\end{align*}
	Since the left-hand-side of \eqref{B-3} is a strictly increasing function of $x$ with its range $(0,~\infty)$, the root of \eqref{B-3} always uniquely exists. (See Lemma \ref{lemma48})\newline

	\noindent(c) With such $\tilde n_1$, we define $\tilde U$ and $\tilde n_{2}, \tilde n_3, \tilde n_4$ as follows:
	\begin{align}\label{B-2}
	\begin{aligned}
	\tilde n_{i}&:=n^{(i)}+\lambda_i \frac{\tilde{\nu}_1}{\tilde{\nu}_i}(\tilde n_1-n^{(1)}),\quad i=2,3,4,\\
	\tilde U&:=\sum_{i=1}^4 \tilde{\nu}_im_in^{(i)}U^{(i)}\bigg/ \sum_{i=1}^4\tilde{\nu}_i m_i n^{(i)}.\\
	\end{aligned}
	\end{align}
	And $\tilde{T}$ is defined by
	\begin{align}\label{Temp}
	\begin{aligned}
	\tilde T:=F(\tilde n_1)=\frac{\bigg\{\sum_{i=1}^4\tilde{\nu}_i n^{(i)}\Big[\frac{1}{2}m_i(|U^{(i)}|^2-|\tilde U|^2)+\frac{3}{2}kT^{(i)}\Big]+\Delta E\tilde\nu_1(\tilde n_{1}-n^{1})\bigg\}}{\bigg(\frac{3}{2}k\sum_{i=1}^4 \tilde\nu_i n^{(i)}\bigg)}.
	\end{aligned}
	\end{align}

\section{Main result}
Before we state our main result, we need to define notations and norms:\newline

\noindent$\bullet$ Every constant denoted by $C$ will be generically defined. The values of $C$ may differ line by line.\\
$\bullet$ We use $C_{l,u}$ to denote a positive constant depending only on the given constants and the quantities defined in (\ref{parameters1}),(\ref{parameters2}) and (\ref{parameters3}).\\
$\bullet$ We define the norm $\norm{\cdot}_{L_2^1}$ by
\begin{align*}
\norm{f}_{L_2^1}=\int_{\mathbb{R}^3}|f(x,v)|(1+|v|^2)dv.
\end{align*}
$\bullet$ We define the following quantities for brevity $(i=1,2,3,4)$.
\begin{align}\label{parameters1}
\begin{split}
&\hspace{0.2cm}a_{i,u}=2\int_{\mathbb{R}^3}f_{i,LR}dv,\quad  a_{i,s}=\int_{\mathbb{R}^3}\frac{1}{|v_1|}f_{i,LR}dv,\quad
a_{i,l}=\frac{1}{8}a_{i,u},\cr
&c_{i,u}=2\int_{\mathbb{R}^3}f_{i,LR}|v|^2dv\quad c_{i,s}=\int_{\mathbb{R}^3}\frac{1}{|v_1|}f_{i,LR}|v|^2dv,
\quad c_{i,l}=\frac{1}{8}c_{i,u},
\end{split}
\end{align}
where we used the notation:
\begin{align*}
f_{i,LR}(v)&=f_{i,L}(v)1_{v_1>0}+f_{i,R}(v)1_{v_1<0}.
\end{align*}
From this, we define
\begin{align}\label{parameters2}
\begin{split}
a_u=\underset{1\leq i\leq 4}{\max}{\{a_{i,u}\}},\ a_l=\underset{1\leq i\leq 4}{\min}{\{a_{i,l}\}},\ c_u=\underset{1\leq i\leq 4}{\max}{\{c_{i,u}\}},\  c_l=\underset{1\leq i\leq 4}{\min}{\{c_{i,l}\}}.\ 
\end{split}
\end{align}
$\bullet$ We also define the following quantity, which will serve as a lower bound for the temperature.
\begin{align}\label{parameters3}
\gamma_{i,l}=\frac{1}{16}\bigg(\int_{v_1>0}f_{i,L}|v_1|dv\bigg)\bigg(\int_{v_1<0}f_{i,R}|v_1|dv\bigg)
\end{align}
and 
\[
\gamma_l=\underset{1\leq i\leq 4}{\min}\{\gamma_{i,l}\}.
\]

Now, we define the mild solution of the system of PDE (\ref{system}):
\begin{definition} A pair of functions $f=(f_1,f_2,f_3,f_4) \in (L^\infty([0,1]_x;L_2^1(\mathbb{R}_v^3)))^4$ is said to be a mild solution for (\ref{system}) if $f_i$ satisfies the following equation:
\begin{align*}
f_i(x,v)=&\bigg(e^{-\frac{1}{\tau |v_i|}\int_0^x \nu_i(y)dy }f_{i,L}(v)+\frac{1}{\tau |v_1|}\int_0^xe^{-\frac{1}{\tau |v_1|}\int_y^x\nu_i(z)dz}\nu_i\mathcal{M}_idy\bigg)1_{v_1>0}\\
&+\bigg(e^{-\frac{1}{\tau |v_1|}\int_x^1 \nu_i(y)dy }f_{i,R}(v)+\frac{1}{\tau |v_1|}\int_x^1e^{-\frac{1}{\tau |v_1|}\int_x^y\nu_i(z)dz}\nu_i\mathcal{M}_idy\bigg)1_{v_1<0},
\end{align*}
for each $i=1,2,3,4$.
\end{definition}
And we also define the mild solution to (\ref{system2}):
\begin{definition} A pair of functions $f=(f_1,f_2,f_3,f_4) \in (L^\infty([0,1]_x;L_2^1(\mathbb{R}_v^3)))^4$ is said to be a mild solution for (\ref{system2}) if $f_i$ satisfies the following equation:
	\begin{align*}
	f_i(x,v)=&\bigg(e^{-\frac{1}{\tau |v_i|}\int_0^x \tilde\nu_i(y)dy }f_{i,L}(v)+\frac{1}{\tau |v_1|}\int_0^xe^{-\frac{1}{\tau |v_1|}\int_y^x\tilde\nu_i(z)dz}\tilde\nu_i\tilde{\mathcal{M}}_idy\bigg)1_{v_1>0}\\
	&+\bigg(e^{-\frac{1}{\tau |v_1|}\int_x^1 \tilde\nu_i(y)dy }f_{i,R}(v)+\frac{1}{\tau |v_1|}\int_x^1e^{-\frac{1}{\tau |v_1|}\int_x^y\tilde\nu_i(z)dz}\tilde\nu_i\tilde{\mathcal{M}}_idy\bigg)1_{v_1<0},
	\end{align*}
	for each $i=1,2,3,4$.
\end{definition}
The main results of this paper are as follows:
\begin{theorem}\label{main} Suppose $f_{i,LR},\frac{1}{|v_1|}f_{i,LR}\in L_2^1(\mathbb{R}^3_v)$. Assume that the inflow data does not induce vertical flows on the boundary: 
\begin{align*}
&\int_{\mathbb{R}^2}f_{i,L}v_jdv_2dv_3=\int_{\mathbb{R}^2}f_{i,R}v_jdv_2dv_3=0.\quad (j=2,3)
\end{align*}
Then there exist two constants $\epsilon,L>0$, depending only the constants defined in (\ref{parameters1}), (\ref{parameters2}) and (\ref{parameters3}), such that if $\epsilon>\nu_{12}^{34}>0$ and $\tau>L$, then there exists a unique mild solution $f=(f_1,f_2,f_3,f_4)$ for (\ref{system}) satisfying
\begin{align*}
a_{i,l}\leq\int_{\mathbb{R}^3}f_i(x,v)dv\leq a_{i,u},\qquad c_{i,l}\leq\int_{\mathbb{R}^3}|v|^2f_i(x,v)dv\leq c_{i,u}
\end{align*}
and
\begin{align*}
\bigg(\int_{\mathbb{R}^3}f_idv\bigg)\bigg(\int_{\mathbb{R}^3}|v|^2f_idv\bigg)-\bigg(\int_{\mathbb{R}^3}v_1f_idv\bigg)^2\geq\gamma_l.
\end{align*}
\end{theorem}
\begin{theorem}\label{main2} Suppose $f_{i,LR},\frac{1}{|v_1|}f_{i,LR}\in L_2^1$. Assume that the inflow data does not induce vertical flows on the boundary:
	\begin{align*}
	&\int_{\mathbb{R}^2}f_{i,L}v_jdv_2dv_3=\int_{\mathbb{R}^2}f_{i,R}v_jdv_2dv_3=0. \quad(j=1,2)
	\end{align*}
	Then there exists a constant $L>0$, depending only the constants defined in (\ref{parameters1}), (\ref{parameters2}) and (\ref{parameters3}), such that if $\tau>L$, then there exists a unique mild solution $f=(f_1,f_2,f_3,f_4)$ for (\ref{system2}) satisfying
	\begin{align*}
	a_{i,l}\leq\int_{\mathbb{R}^3}f_i(x,v)dv\leq a_{i,u},\qquad c_{i,l}\leq\int_{\mathbb{R}^3}|v|^2f_i(x,v)dv\leq c_{i,u}
	\end{align*}
	and
	\begin{align*}
	\bigg(\int_{\mathbb{R}^3}f_idv\bigg)\bigg(\int_{\mathbb{R}^3}|v|^2f_idv\bigg)-\bigg(\int_{\mathbb{R}^3}v_1f_idv\bigg)^2\geq\gamma_l.\cr
	\end{align*}
\end{theorem}
The key difficulty and novelty, along with being able to treat two different types of reactive model in a unified manner, arise from the way in which the reactive equilibrium coefficients are estimated.
The macroscopic fields for non-reactive BGK models, which corresponds to the reactive equilibrium coefficients of the reactive BGK models, are defined from velocity distribution functions in an explicit manner through simple integral relation, and the relevent lower and upper bounds for the fields follows directly from the definition once suitable upper and lower
bounds are known for the moments of the distribution function.
The equilibrium coefficients for the reactive system on the other hand, are defined through highly nonlinear relations as was given in Section 2 above. Thererfore, determining various necessary a priori estimates of them cannot be treated in a straightforward manner as in the non-reactive case.

To overcome this difficulty, we introduce several nonlinear functionals which capture imortant structures of such nonlinear relations and carefully analyze  those functionals to derive the desired results for equilibrium coefficients.
For example, to estimate $T_i$ of the first reactive model of our paper, we introduce the following nonlinear functional:
\begin{align*}
	f(t)=&\frac{2}{\nu_i}\sum_{j=1}^{4}\chi_{ij}\frac{\mu_{ij}}{m_i+m_j}n^{(i)}n^{(j)}m_j|U^{(j)}-U^{(i)}|^2+t\frac{m_i}{2}|U^{(i)}-U|^2\\
	&-\frac{1}{2m_i(n^{(i)}+t)}\bigg|m_it(U-U^{(i)})+\frac{2}{\nu_i}\sum_{j=1}^{4}\chi_{ij}\mu_{ij}n^{(i)}n^{(j)}(U^{(j)}-U^{(i)})\bigg|^2,
\end{align*}
and show that $f$ satisfies 
$f(t)\geq -C|t|$ near zero. This observation enables us to estimate $T_i$ from above and below (See Lemma \ref{lb}).

On the other hand, to obtain the uniform lower and upper bound of $\tilde{n}_1$ of the second model, we 
devise the following nonlinear functional:
\begin{align*}
	F_{\bold x, \bold y, \boldsymbol\mu,\boldsymbol\eta,\boldsymbol\alpha,\boldsymbol\beta}(z)
	&= \log\frac{\mu_3 \mu_4}{\eta_2}+\log z+\log(\mu_2x_2+\mu_1z-\eta_1y_1)-\log(\eta_3y_3-\mu_1z+\eta_1y_1)\\
	&-\log(\eta_4y_4-\mu_1z+\eta_1y_1)-\frac{\frac{3}{2}\Delta E\sum\limits_{i=1}^4 \eta_i y_i}{\sum\limits_{i=1}^4\mu_i x_i\Big[\frac{1}{2}m_i(\beta_i^2)+\frac{3}{2}k\alpha_i\Big]+\Delta E(\mu_1z-\eta_1y_1)},
\end{align*}
and use the fact that for each fixed $z$, the function $(\bold x, \bold y, \boldsymbol\mu,\boldsymbol\eta,\boldsymbol\alpha,\boldsymbol\beta)\mapsto F_{\bold x, \bold y, \boldsymbol\mu,\boldsymbol\eta,\boldsymbol\alpha,\boldsymbol\beta}(z) $ is monotone in $x_i, \mu_i, \alpha_i, |\beta_i|, y_i, \eta_i$ $(i=1,2,3,4)$ (See Lemma \ref{lemma48}).

Before moving on to the proof, a brief review on the relevant analytical results are in order.
  The stationary problem for the BGK model in a bounded interval was first studied by Ukai in \cite{Ukai} using a Schauder type fixed point theorem.  Nouri studied the existence of weak solutions for a qunatum BGK model with a discretized condensation ansantz in \cite{Nouri}. The existence of unique mild solutions were obtained in \cite{Bang Y}, using classical Banach fixed point argument. This argument were then applied to the relativistic BGK model of Marle type \cite{HY} and to the quantum BGK model for non-saturated Fermion system and the Boson system without condensation \cite{BGCY}.

For Boltzmann equation, Ardyred et al. considered the slab problem in the framework of  measure-valued solutions \cite{ACI}. Arkeryd and Nouri studied the existence of weak solutions in a series of papers \cite{AN1,AN2,AN3}, which were extended to gas mixture problems by Brull \cite{B1,B2}. Ghomeshi considered the existence and uniqueness of the Boltzmann equation in a slab in \cite{Ghomeshi} (See also \cite{Maslova}). Esposito et al. \cite{ELM} studied hydrodynamic limits in a slab. All the literature reviewed on the existence is for slab problems. For stationary problems in general domains, we refer to \cite{EGKM,Giu1,Giu2}\newline
\section{Estimates of The Macroscopic Parameters}
Throughout this section, we assume that the velocity distribution $F=(f_1,f_2,f_3,f_4)$ satisfies the following inequalities:
	\begin{align*}
	a_{i,l}\leq\int_{\mathbb{R}^3}f_i(x,v)dv\leq a_{i,u},\qquad c_{i,l}\leq\int_{\mathbb{R}^3}|v|^2f_i(x,v)dv\leq c_{i,u}
	\end{align*}
	and
	\begin{align*}
	\bigg(\int_{\mathbb{R}^3}f_idv\bigg)\bigg(\int_{\mathbb{R}^3}|v|^2f_idv\bigg)-\bigg(\int_{\mathbb{R}^3}v_1f_idv\bigg)^2\geq\gamma_l.
	\end{align*}

\subsection{Estimates of the single component parameters} 
To prove the main theorems, we first estimate the macroscopic parameters.
\begin{lemma}\label{scmp}
	The single component parameters satisfy 
	\begin{align*}
	|U^{(i)}|\leq\frac{a_{i,u}+c_{i,u}}{2a_{i,l}}
	\end{align*}
	and
	\begin{align*}
	\frac{m_i\gamma_l}{3ka_{i,u}^2}\leq T^{(i)} \leq \frac{m_ic_{i,u}}{3ka_{i,l}}.
	\end{align*}
	\begin{proof}
		Firstly, $|U^{(i)}|$ can be written as follows:
		\begin{align*}
		|U^{(i)}|=\frac{|\rho^{(i)}U^{(i)}|}{\rho^{(i)}}=\frac{|\int_{\mathbb{R}^3}vf_idv|}{\int_{\mathbb{R}^3}f_idv}.
		\end{align*}
		By Young's inequality, we have
		\[
		\bigg|\int_{\mathbb{R}^3}vf_idv\bigg|\leq\frac{\int_{\mathbb{R}^3}f_idv+\int_{\mathbb{R}^3}|v|^2f_idv}{2}\leq\frac{a_{i,u}+c_{i,u}}{2},
		\]
		so that
		\begin{align*}
		|U^{(i)}|\leq\frac{a_{i,u}+c_{i,u}}{2a_{i,l}}.
		\end{align*}
		Secondly, $T^{(i)}$ can be expressed as
		\begin{align*}
		T^{(i)}&=\frac{(3kn^{(i)}T^{(i)}+\rho^{(i)}|U^{(i)}|^2)-|\rho^{(i)}U^{(i)}|^2(\rho^{(i)})^{-1}}{3kn^{(i)}}\\
		&=\frac{m_i\int_{\mathbb{R}^3} |v|^2 f_idv - m_i|\int_{\mathbb{R}^3}vf_idv|^2|(\int_{\mathbb{R}^3}f_idv)^{-1}}{3k\int_{\mathbb{R}^3}f_idv}.
		\end{align*}
		For the upper bound, we see that
		\begin{align*}
		T^{(i)}=\frac{m_i\int_{\mathbb{R}^3} |v|^2 f_idv - m_i|\int_{\mathbb{R}^3}vf_idv|^2|(\int_{\mathbb{R}^3}f_idv)^{-1}}{3k\int_{\mathbb{R}^3}f_idv}
		\leq\frac{m_i\int_{\mathbb{R}^3} |v|^2 f_idv}{3k\int_{\mathbb{R}^3}f_idv}
		\leq\frac{m_ic_{i,u}}{3ka_{i,l}}.
		\end{align*}
		For the lower bound, we have
		\begin{align*}
		T^{(i)}=\frac{m_i(\int_{\mathbb{R}^3}f_idv)(\int_{\mathbb{R}^3} |v|^2 f_idv) - m_i|\int_{\mathbb{R}^3}vf_idv|^2}{3k(\int_{\mathbb{R}^3}f_idv)^2}\geq\frac{m_i\gamma_l}{3ka_{i,u}^2}.
		\end{align*}
	\end{proof}
\end{lemma}

\subsection{Macroscopic parameters for first model}
\begin{lemma}\label{gm}
	Global macroscopic parameters $U$ and $T$ satisfy
	\begin{align*}
	|U|\leq \underset{1\leq i\leq 4}{\max{}}\bigg\{\frac{a_{i,u}+c_{i,u}}{2a_{i,l}}\bigg\}
	\end{align*}
	and
	\begin{align*}
	T_l\leq T\leq T_u,
	\end{align*}
	where $T_l:=\underset{1\leq i\leq 4}{\min}\bigg\{\frac{m_i\gamma_l}{3ka_{i,u}^2}\bigg\}$ and $T_u:=\frac{c_u}{12ka_l}\sum_{i=1}^4m_i.$
	\begin{proof} For the bound of $U$, we estimate as follows:
\[
|U|\leq\frac{1}{\rho}\sum_{i=1}^{4}\rho^{(i)}|U^{(i)}|\leq \max_{1\leq i\leq 4}|U^{(i)}|\leq\underset{1\leq i\leq 4}{\max{}}\bigg\{\frac{a_{i,u}+c_{i,u}}{2a_{i,l}}\bigg\}.
\]
		For the lower bound of $T$, we observe 
\[
\sum_{i=1}^4\rho^{(i)}(|U^{(i)}|^2-|U|^2)=\sum_{i=1}^4\rho^{(i)}(|U^{(i)}-U|^2)\geq0,
\] which implies
\[
T\geq \sum_{i=1}^{4}\frac{n^{(i)}}{n}T^{(i)}\geq \min_{1\leq i\leq 4}T^{(i)}\geq T_l.
\]
For the upper bound of $T$, we estimate as follows:
\[
T\leq\sum_{i=1}^{4}\frac{n^{(i)}}{n}T^{(i)}+\frac{1}{3nk}\sum_{i=1}^{4}\rho^{(i)}|U^{(i)}|^2 =\frac{1}{3nk}\sum_{i=1}^4m_i\int_{\mathbb R^3}|v|^2f_idv\leq \frac{c_u}{12ka_l}\sum_{i=1}^4m_i.
\]
	\end{proof}
\end{lemma}
\begin{lemma}\label{nl} There exists a positive lower bound for $n_i$ depeding only on the quantities given in (\ref{parameters1}), (\ref{parameters2}) and (\ref{parameters3}).
	\begin{proof} We consider only the case with $i=1$. By the definition
		\begin{align*}
		n_1&= n^{(1)}-\frac{1}{\nu_1}\nu_{12}^{34}\frac{2}{\sqrt{\pi}}\Gamma\bigg(\frac{3}{2},\frac{\Delta E}{kT}\bigg)n^{(1)}n^{(2)}\\
		&\quad+\frac{1}{\nu_1}\nu_{12}^{34}\frac{2}{\sqrt{\pi}}\Gamma\bigg(\frac{3}{2},\frac{\Delta E}{kT}\bigg)n^{(3)}n^{(4)}\bigg(\frac{m_1m_2}{m_3m_4}\bigg)^{3/2}e^{\Delta E/kT}.
		\end{align*}
		Since 
\[
\bigg(\sum_{j=1}^{4}\nu_{ij}+\frac{2}{\sqrt{\pi}}\Gamma\bigg(\frac{3}{2},\frac{\Delta E}{kT_u}\bigg)\nu_{12}^{34}\bigg)a_u\geq\nu_1\geq \nu_{12}^{34}\frac{2}{\sqrt{\pi}}\Gamma\bigg(\frac{3}{2},\frac{\Delta E}{kT}\bigg)n^{(2)},
\] we obtain
		\begin{align*}
		n_1&\geq \frac{1}{\nu_1}\nu_{12}^{34}\frac{2}{\sqrt{\pi}}\Gamma\bigg(\frac{3}{2},\frac{\Delta E}{kT}\bigg)n^{(3)}n^{(4)}\bigg(\frac{m_1m_2}{m_3m_4}\bigg)^{3/2}e^{\Delta E/kT}\\
		&\geq \frac{2\nu_{12}^{34}}{\sqrt{\pi}(\sum_{j=1}^{4}\nu_{1j}+\frac{2}{\sqrt{\pi}}\Gamma(\frac{3}{2},\frac{\Delta E}{kT_u})\nu_{12}^{34})a_u}\Gamma\bigg(\frac{3}{2},\frac{\Delta E}{kT_l}\bigg)(a_l)^2\bigg(\frac{m_1m_2}{m_3m_4}\bigg)^{3/2}e^{\Delta E/kT_u}
		\end{align*}
	\end{proof}
\end{lemma}
\begin{lemma}\label{ub}
	There exist positive upper bounds for $n_i,|U_i|,$ and $T_i$ depending only on the quantities given in (\ref{parameters1}), (\ref{parameters2}) and (\ref{parameters3}).
	\begin{proof}
		We have 
		\begin{align}\label{su}
		|\mathcal{S}|\leq \nu_{12}^{34}\bigg[\frac{\mu^{12}}{\mu^{34}}e^{\Delta E/kT_l}+1\bigg]a_u^2,
		\end{align}
		and
		\begin{align}\label{nul}
		\nu_i\geq \sum_{j=1}^{4}\nu_{ij}a_l,
		\end{align}
		which directly implies
		\begin{align*}
		n_i\leq a_{i,u}+\frac{1}{\sum_{j=1}^{4}\nu_{ij}a_l}\nu_{12}^{34}\bigg[\frac{\mu^{12}}{\mu^{34}}e^{\Delta E/kT_l}+1\bigg]a_u^2.
		\end{align*}
		For $U_i$, the triangle inequality gives
		\begin{align*}
		|U_i|\leq \frac{n^{(i)}}{n_i}|U^{(i)}|+\frac{2}{m_in_i\nu_i}\sum_{j=1}^{4}\chi_{ij}\mu_{ij}n^{(i)}n^{(j)}(|U^{(j)}|+|U^{(i)}|)+\frac{1}{n_i\nu_i}|U||\mathcal{S}|.
		\end{align*}
		Thus, Lemma \ref{scmp} and \ref{gm},  together with (\ref{su}) and (\ref{nul}) give
		\begin{align*}
		|U_i|\leq C_{l,u}.
		\end{align*}
		And for $T_i$, we have from Lemma \ref{gm} that
		\begin{align*}		T_i&\leq\frac{n^{(i)}}{n_i}T^{(i)}+\frac{m_in^{(i)}}{3kn_i}|U^{(i)}|^2+\frac{4}{n_i\nu_i}\sum_{j=1}^{4}\chi_{ij}\frac{\mu_{ij}}{m_i+m_j}n^{(i)}n^{(j)}T^{(j)}\\
		&\quad +\frac{4}{3kn_i\nu_i}\sum_{j=1}^{4}\chi_{ij}\frac{\mu_{ij}}{m_i+m_j}n^{(i)}n^{(j)}(m_i|U^{(i)}|+m_j|U^{(j)}|)(|U^{(j)}|+|U^{(i)}|)\\
		&\quad +\frac{2}{3kn_i\nu_i}|\mathcal{S}|\bigg[\frac{1}{2}m_i|U|^2+\frac{3}{2}kT+\frac{M-m_i}{M}kT\frac{(\Delta E/kT)^{3/2} e^{-\Delta E/kT}}{\Gamma(\frac{3}{2},\frac{\Delta E}{kT})}+\frac{M-m_i}{M}\Delta E\bigg]\\
		&\leq C_{l,u}.
		\end{align*}
	\end{proof}
\end{lemma}
\begin{lemma}\label{lb}
	There exists a positive number $\epsilon$ depeding only on the quantities defined in (\ref{parameters1}), (\ref{parameters2}) and (\ref{parameters3}) such that if $\epsilon>\nu_{12}^{34}>0$, then  $T_i$ has a positve lower bound depending only on the quantities given in (\ref{parameters1}), (\ref{parameters2}) and (\ref{parameters3}).
	\begin{remark}
	 The smallness of $\nu_{12}^{34}$ in Theorem \ref{main} comes from this lemma.
	\end{remark}
	\begin{proof}
		We define $I_i$ and $II_i$ $(i=1,2,3)$ by
		\begin{align*}
		\frac{3}{2}n_ikT_i&=\frac{3kn^{(i)}}{2}\bigg(1-\frac{4}{\nu_i}\sum_{j=1}^{4}\chi_{ij}\frac{m_im_j}{(m_i+m_j)^2}n^{(j)}\bigg)T^{(i)}+\frac{6k}{\nu_i}\sum_{j=1}^{4}\chi_{ij}\frac{\mu_{ij}}{m_i+m_j}n^{(i)}n^{(j)}T^{(j)}\\
		&\quad +\frac{\lambda_i}{\nu_i}\mathcal{S}\bigg[ \frac{3}{2}kT+\frac{M-m_i}{M}\Delta E\bigg(\frac{(\Delta E/kT)^{1/2} e^{-\Delta E/kT}}{\Gamma(\frac{3}{2},\frac{\Delta E}{kT})}-\frac{1-\lambda_i}{2}\bigg) \bigg]\\
		&\quad +\frac{2}{\nu_i}\sum_{j=1}^{4}\chi_{ij}\frac{\mu_{ij}}{m_i+m_j}n^{(i)}n^{(j)}(m_iU^{(i)}+m_jU^{(j)})(U^{(j)}-U^{(i)})\\
		&\quad-\frac{1}{2}m_i[n_iU_i^2-n^{(i)}|U^{(i)}|^2]+\frac{\lambda_i}{\nu_i}\mathcal{S}\frac{1}{2}m_i|U|^2\\
		&= I_1+I_2+I_3+II_1+II_2+II_3
		\end{align*}
\noindent (1) Estimate of $I_1+I_2+I_3$: Since
		\begin{align*}
		\nu_i\geq \sum_{j=1}^{4}\nu_{ij}n^{(j)}
		\geq \sum_{j=1}^{4}\chi_{ij}n^{(j)}
		\geq 4\sum_{j=1}^{4}\chi_{ij}\frac{m_im_j}{(m_i+m_j)^2}n^{(j)},
		\end{align*}
		we have $I_1\geq0.$\\
For the estimate of $I_2+I_3$, we observe from  Lemma \ref{gm} that 
\begin{align*}
I_3\geq& -\frac{1}{\nu_i}|\mathcal{S}|\bigg[ \frac{3}{2}kT+\frac{M-m_i}{M}\Delta E\bigg(\frac{(\Delta E/kT)^{1/2} e^{-\Delta E/kT}}{\Gamma(\frac{3}{2},\frac{\Delta E}{kT})}+1\bigg) \bigg]\\
&\geq -\frac{1}{\nu_i}\nu_{12}^{34}\bigg[\frac{\mu^{12}}{\mu^{34}}e^{\Delta E/kT_l}+1\bigg]a_u^2\bigg[ \frac{3}{2}kT_u+\frac{M-m_i}{M}\Delta E\bigg(\frac{(\Delta E/kT_l)^{1/2} e^{-\Delta E/kT_u}}{\Gamma(\frac{3}{2},\frac{\Delta E}{kT_l})}+1\bigg) \bigg],
\end{align*}		
so that,  for sufficiently small $\nu_{12}^{34}$, we have
		\begin{align*}
		&I_2+I_3\cr &\quad\geq\frac{6k}{\nu_i}\sum_{j=1}^{4}\chi_{ij}\frac{\mu_{ij}}{m_i+m_j}(a_l)^2\frac{m_i\gamma_l}{(a_u)^2}\\ &\quad-\frac{1}{\nu_i}\nu_{12}^{34}\bigg[\frac{\mu^{12}}{\mu^{34}}e^{\Delta E/kT_l}+1\bigg]a_u^2\bigg[ \frac{3}{2}kT_u+\frac{M-m_i}{M}\Delta E\bigg(\frac{(\Delta E/kT_l)^{1/2} e^{-\Delta E/kT_u}}{\Gamma(\frac{3}{2},\frac{\Delta E}{kT_l})}+1\bigg) \bigg]\\
		&\quad\geq\frac{3k}{\nu_i}\sum_{j=1}^{4}\chi_{ij}\frac{\mu_{ij}}{m_i+m_j}(a_l)^2\frac{m_i\gamma_l}{(a_u)^2}.
		\end{align*}

\noindent(2) Estimate of $II_1+II_2+II_3$: By straightforward computations, we get
		\begin{align}\label{in view 1}
		\begin{split}
		II_1+II_2+II_3=&\frac{2}{\nu_i}\sum_{j=1}^{4}\chi_{ij}\frac{\mu_{ij}}{m_i+m_j}n^{(i)}n^{(j)}m_j|U^{(j)}-U^{(i)}|^2\\
		&-\frac{m_in_i}{2}|U_i-U^{(i)}|^2+\frac{\lambda^i}{\nu_i}\mathcal{S}\frac{m_i}{2}|U^{(i)}-U|^2.
		\end{split}
		\end{align}
		To estimate this, we note that the second term in  (\ref{in view 1})

		\begin{align}\label{in view 2}
		\begin{split}
		\frac{m_in_i}{2}|U_i-U^{(i)}|^2
		&=\frac{1}{2m_in_i}|m_in_iU_i-m_in_iU^{(i)}|^2\\
&=\frac{1}{2m_in_i}\Big|m_in_iU_i-m_i n^{(i)} U^{(i)}-m_i \frac{\lambda_i}{\nu_i}\mathcal{S} U^{(i)}\Big|^2\\
		&=\frac{1}{2m_in_i}\bigg|m_i\frac{\lambda^i}{\nu_i}\mathcal{S}(U-U^{(i)})+\frac{2}{\nu_i}\sum_{j=1}^{4}\chi_{ij}\mu_{ij}n^{(i)}n^{(j)}(U^{(j)}-U^{(i)})\bigg|^2,
		\end{split}
		\end{align}
where $\eqref{B-1}_1$ and $\eqref{B-1}_2$ were used in the second and the third line.
		In view of (\ref{in view 1}) and (\ref{in view 2}), we define a function $f$ by
		\begin{align*}
		f(t)=&\frac{2}{\nu_i}\sum_{j=1}^{4}\chi_{ij}\frac{\mu_{ij}}{m_i+m_j}n^{(i)}n^{(j)}m_j|U^{(j)}-U^{(i)}|^2+t\frac{m_i}{2}|U^{(i)}-U|^2\\
		&-\frac{1}{2m_i(n^{(i)}+t)}\bigg|m_it(U-U^{(i)})+\underbracket{\frac{2}{\nu_i}\sum_{j=1}^{4}\chi_{ij}\mu_{ij}n^{(i)}n^{(j)}(U^{(j)}-U^{(i)})}_{A}\bigg|^2.
		\end{align*}
		Then we employ Theorem 3.1 in \cite{AAP} to find $f(0)\geq0$. Moreover, by differentiating this function, we get
		\begin{align*}
		f'(t)=&\frac{m_i}{2}|U^{(i)}-U|^2+\frac{1}{2m_i(n^{(i)}+t)^2}\bigg|m_it(U-U^{(i)})+A\bigg|^2\\
		&-\frac{U-U^{(i)}}{(n^{(i)}+t)}\cdot \bigg(m_it(U-U^{(i)})+A\bigg)\\
		=&\frac{1}{2m_i(n^{(i)}+t)^2}|m_in^{(i)}(U-U^{(i)})+A|^2.
		\end{align*}
		Since $|f'(t)|\leq C_{l,u}$ on $-a_{i,l}/2\leq t\leq a_{i,l}/2$, we obtain
		\begin{align*}
		f(t)\geq f(0)-C_{l,u}|t|\geq -C_{l,u}|t|.
		\end{align*}
	 Therefore,
		\begin{align*}
		\frac{3}{2}n_ikT_i&\geq \frac{3k}{\nu_i}\sum_{j=1}^{4}\chi_{ij}\frac{\mu_{ij}}{m_i+m_j}(a_l)^2\frac{m_i\gamma_l}{(a_u)^2}+f\Big(\frac{\lambda^i\mathcal{S}}{\nu_i}\Big)\\
		&\geq \frac{3k}{\nu_i}\sum_{j=1}^{4}\chi_{ij}\frac{\mu_{ij}}{m_i+m_j}(a_l)^2\frac{m_i\gamma_l}{(a_u)^2}-C_{l,u}\bigg|\frac{\lambda^i\mathcal{S}}{\nu_i}\bigg|\\
		&\geq
		\frac{3k}{2\nu_i}\sum_{j=1}^{4}\chi_{ij}\frac{\mu_{ij}}{m_i+m_j}(a_l)^2\frac{m_i\gamma_l}{(a_u)^2},
		\end{align*}
		where we used that (\ref{su}), (\ref{nul}) imply that there exist a positive number $\epsilon$ depending only on the quantities defined in (\ref{parameters1}), (\ref{parameters2}) and (\ref{parameters3}) such that if $\epsilon>\nu_{12}^{34}>0$, then
		\begin{align*}
		\bigg|\frac{\lambda^i\mathcal{S}}{\nu_i}\bigg|\leq\min\bigg\{\frac{a_{i,l}}{2},\frac{3k}{2C_{l,u}\nu_i}\sum_{j=1}^{4}\chi_{ij}\frac{\mu_{ij}}{m_i+m_j}(a_l)^2\frac{m_i\gamma_l}{(a_u)^2}\bigg\}.
		\end{align*} 
		
		Finally, note that if $\epsilon>\nu_{12}^{34}>0$, then for $i=1,2$,
		\begin{align*}
		\nu_{i}&\leq \sum_{j=1}^4 \nu_{1j}a_u+\frac{2}{\sqrt{\pi}}\Gamma\bigg(\frac{3}{2},\frac{\Delta E}{kT_u}\bigg)\nu_{12}^{34}a_u
		\leq C_{l,u},
		\end{align*}
		and for $i=3,4$,
		\begin{align*}
		\nu_i&\leq\sum_{j=1}^4 \nu_{3j}a_u+\frac{2}{\sqrt{\pi}}\Gamma\bigg(\frac{3}{2},\frac{\Delta E}{kT_u}\bigg)\bigg(\frac{\mu^{12}}{\mu^{34}}\bigg)^{3/2}e^{\Delta E/kT_l}\nu_{12}^{34}a_u
		\leq C_{l,u},
		\end{align*}
		which completes the proof.
	\end{proof}
\end{lemma}
\subsection{Macroscopic parameter for the second model}
\begin{lemma} There exist postive bounds for $|\tilde U|$ and $\tilde \nu_i$ depending only on the quantities given in (\ref{parameters1}), (\ref{parameters2}) and (\ref{parameters3}).
	\begin{proof} We obtain from Lemma \ref{scmp} that
		\[
		|\tilde U| \leq \sum_{i=1}^4 \nu_im_in^{i}|U^{i}|\bigg/ \sum_{i=1}^4\nu_i m_i n^{i}\leq \max_{1\leq i\leq 4}|U^{i}|\leq \underset{1\leq i\leq 4}{\max{}}\bigg\{\frac{a_{i,u}+c_{i,u}}{2a_{i,l}}\bigg\}=:R,
		\]
		and
		\begin{align*}
		&\nu_1^m:=\sum_{j=1}^4 \nu_{1j}a_{j,l}\leq \tilde\nu_1\leq\sum_{j=1}^4 \nu_{1j}a_{i,u}+\bigg(\frac{\mu_{34}}{\mu_{12}}\bigg)^{3/2}\nu_{34}^{12}a_{2,u}=:\nu_1^M,\\
		&\nu_3^m:=\sum_{j=1}^4 \nu_{3j}a_{j,l}\leq\tilde\nu_3\leq\sum_{j=1}^4 \nu_{3j}a_{j,u}+ \nu_{34}^{12}a_{4,u}=:\nu_3^M.
		\end{align*}
		By the same way, we have
		$\nu_2^m\leq\tilde\nu_2\leq\nu_2^M$ and $\nu_4^m\leq\tilde\nu_4\leq\nu_4^M.$
	\end{proof}
\end{lemma}
\begin{lemma}\label{lemma48} There exist positive lower and upper  bounds for $\tilde n_1$  depending only on the quantities given in (\ref{parameters1}), (\ref{parameters2}) and (\ref{parameters3}).
	\begin{proof}
We first set 
\[
\bold x=(x_1,x_2,x_3,x_4),\quad \bold y=(y_1,y_2,y_3,y_4),\quad \boldsymbol \mu=(\mu_1,\mu_2,\mu_3,\mu_4),
\]
\[
\boldsymbol \nu=(\nu_1,\nu_2,\nu_3,\nu_4),\quad \boldsymbol \alpha=(\alpha_1,\alpha_2,\alpha_3,\alpha_4),\quad \boldsymbol \beta=(\beta_1,\beta_2,\beta_3,\beta_4).
\]
Then, for each $(\bold x, \bold y, \boldsymbol\mu,\boldsymbol\eta,\boldsymbol\alpha,\boldsymbol\beta)\in (\mathbb{R}_+)^{20}\times\mathbb{R}^4$, we define a map $F_{\bold x, \bold y, \boldsymbol\mu,\boldsymbol\eta,\boldsymbol\alpha,\boldsymbol\beta}:\Omega_{\bold x, \bold y, \boldsymbol\mu,\boldsymbol\eta,\boldsymbol\alpha,\boldsymbol\beta}\rightarrow(-\infty,\infty)$ by
		\begin{align*}
		F_{\bold x, \bold y, \boldsymbol\mu,\boldsymbol\eta,\boldsymbol\alpha,\boldsymbol\beta}(z)
		&= \log\frac{\mu_3 \mu_4}{\eta_2}+\log z+\log(\mu_2x_2+\mu_1z-\eta_1y_1)-\log(\eta_3y_3-\mu_1z+\eta_1y_1)\\
		&-\log(\eta_4y_4-\mu_1z+\eta_1y_1)-\frac{\frac{3}{2}\Delta E\sum\limits_{i=1}^4 \eta_i y_i}{\sum\limits_{i=1}^4\mu_i x_i\Big[\frac{1}{2}m_i(\beta_i^2)+\frac{3}{2}k\alpha_i\Big]+\Delta E(\mu_1z-\eta_1y_1)},
		\end{align*}
		where the domain $\Omega_{\bold x, \bold y, \boldsymbol\mu,\boldsymbol\eta,\boldsymbol\alpha,\boldsymbol\beta}$ is given by
		\begin{align*}
		\Omega_{\bold x, \bold y, \boldsymbol\mu,\boldsymbol\eta,\boldsymbol\alpha,\boldsymbol\beta}=&\{z>0\}\cap\{\mu_2x_2+\mu_1z-\eta_1y_1>0\}\cap\{\eta_3y_3-\mu_1z+\eta_1y_1>0\}\\&\cap\{\eta_4y_4-\mu_1z+\eta_1y_1>0\}\\
		&\cap\bigg\{\sum\limits_{i=1}^4\mu_i x_i\Big[\frac{1}{2}m_i(\beta_i^2)+\frac{3}{2}k\alpha_i\Big]+\Delta E(\mu_1z-\eta_1y_1)>0 \bigg\}.
		\end{align*}
We mention that $\Omega_{\bold x, \bold y, \boldsymbol\mu,\boldsymbol\eta,\boldsymbol\alpha,\boldsymbol\beta}$ is always non-empty since $\eta_1y_1/\mu_1$ always belongs to $\Omega_{\bold x, \bold y, \boldsymbol\mu,\boldsymbol\eta,\boldsymbol\alpha,\boldsymbol\beta}$.
	
		We note that, for each $(\bold x, \bold y, \boldsymbol\mu,\boldsymbol\eta,\boldsymbol\alpha,\boldsymbol\beta)\in (\mathbb{R}_+)^{20}\times\mathbb{R}^4$, $F_{\bold x, \bold y, \boldsymbol\mu,\boldsymbol\eta,\boldsymbol\alpha,\boldsymbol\beta}$ is a strictly increasing surjective function with respect to $z$ on $\Omega_{\bold x, \bold y, \boldsymbol\mu,\boldsymbol\eta,\boldsymbol\alpha,\boldsymbol\beta}$. Also, for fixed $z$, the function $(\bold x, \bold y, \boldsymbol\mu,\boldsymbol\eta,\boldsymbol\alpha,\boldsymbol\beta)\mapsto F_{\bold x, \bold y, \boldsymbol\mu,\boldsymbol\eta,\boldsymbol\alpha,\boldsymbol\beta}(z)$ is decreasing in $x_i,\mu_i,\alpha_i,|\beta_i|$ $(i=1,2,3,4)$, and increasing in $y_i,\eta_i$ $(i=1,2,3,4)$, as long as the function is well-defined.\newline
Now, taking logarithms on both sides of \eqref{B-3} yields (with $x$ replaced by $\tilde n_1$)
\[
F_{\bold n,\bold n,\boldsymbol{\tilde\nu},\boldsymbol{\tilde\nu},\bold T,\bold V}(\tilde n_1)=\frac{3}{2}\log\bigg(\frac{\mu^{12}}{\mu^{34}}\bigg),
\]
where we used the following notations:
\begin{align*}
&\bold n=(n^{(i)})_{i=1,2,3,4},\ \boldsymbol{\tilde\nu}=(\tilde \nu_i)_{i=1,2,3,4},\bold T=(T^{(i)})_{i=1,2,3,4},\bold{V}=(|U^{(i)}-\tilde U|^2)_{i=1,2,3,4}.
\end{align*}		
Therefore, there exists the unique function $G:(\mathbb{R}_+)^{20}\times\mathbb{R}^4\rightarrow\mathbb{R}_+$ satisfying
		\begin{align*}
		G(\bold x, \bold y, \boldsymbol\mu,\boldsymbol\eta,\boldsymbol\alpha,\boldsymbol\beta)=F_{\bold x, \bold y, \boldsymbol\mu,\boldsymbol\eta,\boldsymbol\alpha,\boldsymbol\beta}^{-1}\bigg(\frac{3}{2}\log\bigg(\frac{\mu_{12}}{\mu_{34}}\bigg)\bigg)
		\end{align*}
		and furthermore, $G$ is decreasing in $x_i,\mu_i,\alpha_i,|\beta_i|$ $(i=1,2,3,4)$, and increasing in $y_i,\eta_i$ $(i=1,2,3,4)$.
		Therefore, we obtain
		\begin{align*}
		0<G(\bold a_u,\bold a_l,\boldsymbol{\nu^M},\boldsymbol{\nu^m},T_u\bold1,2R\bold 1)\leq \tilde n_1=G(\bold n,\bold n,\boldsymbol{\tilde\nu},\boldsymbol{\tilde\nu},\bold T,\bold V)\leq G(\bold a_l,\bold a_u,\boldsymbol{\nu^m},\boldsymbol{\nu^M},T_l\bold1,\bold0)
		\end{align*}
		where we used the following notations:
		\begin{align*}
		&\bold a_l=(a_{i,l})_{i=1,2,3,4},\ \bold a_u=(a_{i,u})_{i=1,2,3,4},\ \boldsymbol{\nu^m}=(\nu^m)_{i=1,2,3,4},\ \boldsymbol{\nu^M}=(\nu^M)_{i=1,2,3,4},\\
		&\bold n=(n^{(i)})_{i=1,2,3,4},\ \boldsymbol{\tilde\nu}=(\tilde \nu_i)_{i=1,2,3,4},\bold T=(T^{(i)})_{i=1,2,3,4},\bold 1=(1,1,1,1), \bold0=(0,0,0,0),\\
		&\bold{V}=(|U^{(i)}-\tilde U|^2)_{i=1,2,3,4}.
		\end{align*}
	\end{proof}
\end{lemma}
\begin{corollary}
	There exist positive  lower and upper bounds for $\tilde n_2$  depending only on the quantities given in (\ref{parameters1}), (\ref{parameters2}) and (\ref{parameters3}).
	\begin{proof}
		By the definition (\ref{B-2}), we know $\tilde\nu_2(\tilde n_2-n^{(2)})=\tilde\nu_1(\tilde n_1-n^{(1)})$, which implies the following equality:
		\begin{align*}
		\frac{\tilde\nu_3\tilde\nu_4}{\tilde\nu_1\tilde\nu_2}\frac{\tilde\nu_2\tilde n_{2}[ \tilde\nu_1n^{(1)}+\tilde\nu_2(\tilde n_{2}- n^{(2)})]}{[\tilde\nu_3n^{(3)}-\tilde\nu_2(\tilde n_{2}-n^{(2)})][ \tilde\nu_4 n^{4}-\tilde\nu_2( \tilde n_{2}- n^{(2)})]}\exp\left(-\frac{\Delta E}{k\tilde T(\tilde n_{2})}\right)=\left(\frac{\mu^{12}}{\mu^{34}}\right)^{3/2}
		\end{align*}
		and
		\begin{align*}
		\tilde T(\tilde n_2)=\bigg\{\sum_{i=1}^4\tilde{\nu}_i n^{(i)}\Big[\frac{1}{2}m_i(|U^{(i)}|^2-|\tilde U|^2)+\frac{3}{2}kT^{(i)}\Big]+\Delta E\tilde\nu_2(\tilde n_{2}-n^{2})\bigg\}\bigg/\bigg(\frac{3}{2}k\sum_{i=1}^4 \tilde\nu_i n^{(i)}\bigg).
		\end{align*}
		Repeating the argument used in Lemma \ref{lemma48}, we obtain the desired result.
	\end{proof}
\end{corollary}
\begin{corollary} There exist positive lower and upper  bounds for $\tilde n_3$ and $\tilde n_4$  depending only on the quantities given in (\ref{parameters1}), (\ref{parameters2}) and (\ref{parameters3}).
	\begin{proof}
		We rewrite  (\ref{B-2}) as 
		\begin{align}\label{note that}
		\tilde\nu_3(\tilde n_3-n^{(3)})=-\tilde\nu_1(\tilde n_1-n^{(1)})
		\end{align}
		and plug this into  (\ref{B-33}) to get
		\begin{align*}
		\frac{\tilde\nu_3\tilde\nu_4}{\tilde\nu_1\tilde\nu_2}\frac{[\tilde\nu_1n^{(1)}-\tilde\nu_3(\tilde n_{3}- n^{(3)})][\tilde\nu_2n^{(2)}-\tilde\nu_3(\tilde n_{3}- n^{(3)})]}{\tilde\nu_3\tilde n_3[\tilde\nu_4n^{(4)}+\tilde\nu_3(\tilde n_{3}- n^{(3)})]}\exp\left(-\frac{\Delta E}{k\tilde T(\tilde n_{3})}\right)=\left(\frac{\mu^{12}}{\mu^{34}}\right)^{3/2},
		\end{align*}
	where $\tilde T(\tilde n_3)$ is defined by
		\begin{align*}
		\tilde T(\tilde n_3):=\bigg\{\sum_{i=1}^4\tilde{\nu}_i n^{(i)}\Big[\frac{1}{2}m_i(|U^{(i)}|^2-|\tilde U|^2)+\frac{3}{2}kT^{(i)}\Big]-\Delta E\tilde\nu_3(\tilde n_{3}-n^{3})\bigg\}\bigg/\bigg(\frac{3}{2}k\sum_{i=1}^4 \tilde\nu_i n^{(i)}\bigg).
		\end{align*}
	Note that 	$\tilde T(\tilde n_3)$ is obtained by inserting (\ref{note that}) into (\ref{Temp}).
	 
		Now, for each $(\bold x, \bold y, \boldsymbol\mu,\boldsymbol\eta,\boldsymbol\alpha,\boldsymbol\beta)\in (\mathbb{R}_+)^{20}\times\mathbb{R}^4$, we define a map $H_{\bold x, \bold y, \boldsymbol\mu,\boldsymbol\eta,\boldsymbol\alpha,\boldsymbol\beta}:\Omega'_{\bold x, \bold y, \boldsymbol\mu,\boldsymbol\eta,\boldsymbol\alpha,\boldsymbol\beta}\rightarrow(-\infty,\infty)$ by
		\begin{align*}
		H_{\bold x, \bold y, \boldsymbol\mu,\boldsymbol\eta,\boldsymbol\alpha,\boldsymbol\beta}(z)
		&= \log\frac{ \mu_4}{\eta_1 \eta_2}-\log z+\log(\mu_1x_1-\eta_3z+\mu_3x_3)+\log(\mu_2x_2-\eta_3z+\mu_3x_3)\\
		&-\log(\eta_4y_4+\eta_3z-\mu_3x_3)-\frac{\frac{3}{2}\Delta E\sum\limits_{i=1}^4 \eta_i y_i}{\sum\limits_{i=1}^4\mu_i x_i\Big[\frac{1}{2}m_i(\beta_i^2)+\frac{3}{2}k\alpha_i\Big]-\Delta E(\eta_3z-\mu_3x_3)},
		\end{align*}
		where 
		\begin{align*}
		\Omega'_{\bold x, \bold y, \boldsymbol\mu,\boldsymbol\eta,\boldsymbol\alpha,\boldsymbol\beta}=&\{z>0\}\cap\{\mu_1x_1-\eta_3z+\mu_3x_3>0\}\cap\{\mu_2x_2-\eta_3z+\mu_3x_3>0\}\\&\cap\{\eta_4y_4+\eta_3z-\mu_3x_3>0\}\\&\cap
		\bigg\{\sum\limits_{i=1}^4\mu_i x_i\Big[\frac{1}{2}m_i(\beta_i^2)+\frac{3}{2}k\alpha_i\Big]-\Delta E(\eta_3z-\mu_3x_3)>0\bigg\}.
		\end{align*}
		For each $(\bold x, \bold y, \boldsymbol\mu,\boldsymbol\eta,\boldsymbol\alpha,\boldsymbol\beta)\in (\mathbb{R}_+)^{20}\times\mathbb{R}^4$, $H_{\bold x, \bold y, \boldsymbol\mu,\boldsymbol\eta,\boldsymbol\alpha,\boldsymbol\beta}$ is a strictly decreasing surjective function on $\Omega_{\bold x, \bold y, \boldsymbol\mu,\boldsymbol\eta,\boldsymbol\alpha,\boldsymbol\beta}$. Hence there exists the unique function $J:(\mathbb{R}_+)^{20}\times\mathbb{R}^4\rightarrow\mathbb{R}_+$ satisfying
		\begin{align*}
		J(\bold x, \bold y, \boldsymbol\mu,\boldsymbol\eta,\boldsymbol\alpha,\boldsymbol\beta)=H_{\bold x, \bold y, \boldsymbol\mu,\boldsymbol\eta,\boldsymbol\alpha,\boldsymbol\beta}^{-1}\bigg(\frac{3}{2}\log\bigg(\frac{\mu_{12}}{\mu_{34}}\bigg)\bigg).
		\end{align*}
		As in the proof of Lemma \ref{lemma48}, $J$ is increasing in $x_i,\mu_i,\alpha_i,|\beta_i|$ $(i=1,2,3,4)$, and decreasing in $y_i,\eta_i$ $(i=1,2,3,4)$.  Hence the following inequality holds:
\begin{align*}
		0<J(\bold a_l,\bold a_u,\boldsymbol{\nu^m},\boldsymbol{\nu^M},T_u\bold1,2R\bold 1)\leq \tilde n_3\leq J(\bold a_u,\bold a_l,\boldsymbol{\nu^M},\boldsymbol{\nu^m},T_l\bold1,\bold0).
		\end{align*}
	
		The proof for $\tilde n_4$ is similar. We omit it.
	\end{proof}
\end{corollary}
\begin{lemma}
	There exist positive lower and upper bounds for $\tilde T$  depending only on the quantities given in (\ref{parameters1}), (\ref{parameters2}) and (\ref{parameters3}).
	\begin{proof}
		From $\eqref{Temp}$ we get the following relation:
			\begin{align}\label{TA-B}
		&\tilde T\times A-B=\tilde\nu_1(\tilde n_{1}-n^{1}),
		\end{align}
		where
		\begin{align*}
		A&=\frac{3}{2}k\sum_{i=1}^4 \tilde\nu_i n^{(i)}/\Delta E,\cr
		B&=\sum_{i=1}^4\tilde{\nu}_i n^{(i)}\Big[\frac{1}{2}m_i(|U^{(i)}-\tilde U|^2)+\frac{3}{2}kT^{(i)}\Big]/\Delta E.
		\end{align*}
		Plugging (\ref{TA-B}) into (\ref{B-33}), we have
		\begin{align*}
		\frac{\tilde\nu_3\tilde\nu_4}{\tilde\nu_1\tilde\nu_2}\frac{[\tilde\nu_1n^{(1)}+\tilde T\times A-B][\tilde\nu_2n^{(2)}+\tilde T\times A-B]}{[\tilde\nu_3n^{(3)}-\tilde T\times A+B][\tilde\nu_4n^{(4)}-\tilde T\times A+B]}\exp{\bigg(-\frac{\Delta E}{k\tilde T}\bigg)}=\bigg(\frac{\mu_{12}}{\mu_{34}}\bigg)^{3/2}.
		\end{align*}
		For each $(x_1,x_2,y_3,y_4,\boldsymbol{\mu},\boldsymbol{\eta},a, b)\in( \mathbb R_+)^{14}$, we define a strictly increasing function $K:\Lambda_{x_1,x_2,y_3,y_4,\boldsymbol{\mu},\boldsymbol{\eta},a, b}\rightarrow(0,\infty)$ by
		\begin{align*}
		K_{x_1,x_2,y_3,y_4,\boldsymbol{\mu},\boldsymbol{\eta},a, b}(z):=\frac{\mu_3\mu_4}{\eta_1\eta_2}\frac{(\mu_1x_1+az-b)(\mu_2x_2+az-b)}{(\eta_3y_{3}-az+b)(\eta_4y_{4}-az+b)}\exp{\bigg(-\frac{\Delta E}{k  z}\bigg)}
		\end{align*}
		where 
		\begin{align*}
		\Lambda_{\boldsymbol{\mu},\boldsymbol{\nu},\boldsymbol{\alpha},\boldsymbol{\beta},a, b}=&\{x>0\}\cap\{\mu_1x_1+az-b\}\cap\{\mu_2x_2+az-b\}\cap\{\eta_3y_{3}-az+b\}\\
		&\cap\{\eta_4y_{4}-az+b\}.
		\end{align*}
Note that there exist positive constants $A_l,A_u,B_l,B_u$ depending only on the parameters given in (\ref{parameters1}), (\ref{parameters2}) and (\ref{parameters3}) such that $A_l\leq A\leq A_u$ and $B_l\leq B\leq B_u$.
		In the same way as in the proof of Lemma \ref{lemma48}, we have the following inequality:
\begin{align*}
		0<K_{a_u,a_u,a_l,a_l,\boldsymbol{\nu^M},\boldsymbol{\nu^m},A_u,B_l}^{-1}\bigg(\Big(\frac{\mu_{12}}{\mu_{34}}\Big)^{3/2}\bigg)\leq \tilde T\leq K_{a_l,a_l,a_u,a_u,\boldsymbol{\nu^m},\boldsymbol{\nu^M},A_l,B_u}^{-1}\bigg(\Big(\frac{\mu_{12}}{\mu_{34}}\Big)^{3/2}\bigg).
		\end{align*}
	\end{proof}
\end{lemma}
%
%
%
%
\section{Fixed point set-up}
We prove our main theorem applying Banach fixed point theorm to a solution operator defined from the mild form in an appropriately constructed solution space.
We define our solution space as follows:
\begin{align*}
\Omega=\bigg\{f=(f_1,f_2,f_3,f_4)\in(L^\infty([0,1]_x;L_2^1(\mathbb{R}_v^3)))^4\ |\ f_i\ \text{satisfies}\  (\mathcal{A}),(\mathcal{B}),(\mathcal{C})\bigg\}
\end{align*}
with the metric $\displaystyle d(f,g)=\sum_{i=1}^{4}\underset{x\in[0,1]}{\sup}||f_i-g_i||_{L_2^1}$, where $(\mathcal{A}),(\mathcal{B}),$ and $(\mathcal{C})$ denote

 $(\mathcal{A})$ $f_i$ are non-negative.

 $(\mathcal{B})$ The macroscopic quantities satisfy the followings:
\begin{align*}
a_{i,l}\leq\int_{\mathbb{R}^3}f_i(x,v)dv\leq a_{i,u}, \quad c_{i,l}\leq\int_{\mathbb{R}^3}|v|^2f_i(x,v)dv\leq c_{i,u},
\end{align*}

 $(\mathcal{C})$ The following lower bound holds:
\begin{align*}
\bigg(\int_{\mathbb{R}^3}f_idv\bigg)\bigg(\int_{\mathbb{R}^3}|v|^2f_idv\bigg)-\bigg(\int_{\mathbb{R}^3}vf_idv\bigg)^2\geq\gamma_l.
\end{align*}

And we define our solution map for the slow reaction model (\ref{system}) $\Phi:\Omega\rightarrow\Phi(\Omega)$ by $\Phi(f_1,f_2,f_3,f_4)=(\phi_1,\phi_2,\phi_3,\phi_4)$ where $\phi_i$ is defined as below
\begin{align*}
\phi_i(x,v)=&\bigg(e^{-\frac{1}{\tau |v_1|}\int_0^x \nu_i(y)dy }f_{i,L}(v)+\frac{1}{\tau |v_1|}\int_0^xe^{-\frac{1}{\tau |v_1|}\int_y^x\nu_i(z)dz}\nu_i\mathcal{M}_idy\bigg)1_{v_1>0}\\
&+\bigg(e^{-\frac{1}{\tau |v_1|}\int_x^1 \nu_i(y)dy }f_{i,R}(v)+\frac{1}{\tau |v_1|}\int_x^1e^{-\frac{1}{\tau |v_1|}\int_x^y\nu_i(z)dz}\nu_i\mathcal{M}_idy\bigg)1_{v_1<0}.
\end{align*}
For simplicity, we denote $\phi_i=\phi^+_i+\phi^-_i$, where
\begin{align*}
\phi^+_i(x,v)&=e^{-\frac{1}{\tau |v_1|}\int_0^x \nu_i(y)dy }f_{i,L}(v)+\frac{1}{\tau |v_1|}\int_0^xe^{-\frac{1}{\tau |v_1|}\int_y^x\nu_i(z)dz}\nu_i\mathcal{M}_idy,\\
\phi^-_i(x,v)&=e^{-\frac{1}{\tau |v_1|}\int_x^1 \nu_i(y)dy }f_{i,R}(v)+\frac{1}{\tau |v_1|}\int_x^1e^{-\frac{1}{\tau |v_1|}\int_x^y\nu_i(z)dz}\nu_i\mathcal{M}_idy.
\end{align*}
In a similar manner, we define our solution map for the fast reaction model (\ref{system2}) $\widetilde\Phi:\Omega\rightarrow\widetilde\Phi(\Omega)$ by $\widetilde\Phi(f_1,f_2,f_3,f_4)=(\widetilde\phi_1,\widetilde\phi_2,\widetilde\phi_3,\widetilde\phi_4)$ where $\widetilde\phi_i$ is defined as below
\begin{align*}
\widetilde\phi_i(x,v)=&\bigg(e^{-\frac{1}{\tau |v_1|}\int_0^x \tilde\nu_i(y)dy }f_{i,L}(v)+\frac{1}{\tau |v_1|}\int_0^xe^{-\frac{1}{\tau |v_1|}\int_y^x\tilde\nu_i(z)dz}\tilde\nu_i\widetilde{\mathcal{M}}_idy\bigg)1_{v_1>0}\\
&+\bigg(e^{-\frac{1}{\tau |v_1|}\int_x^1 \tilde\nu_i(y)dy }f_{i,R}(v)+\frac{1}{\tau |v_1|}\int_x^1e^{-\frac{1}{\tau |v_1|}\int_x^y\tilde\nu_i(z)dz}\tilde\nu_i\widetilde{\mathcal{M}}_idy\bigg)1_{v_1<0}.
\end{align*}
For simplicity, we denote $\widetilde\phi_i=\widetilde\phi^+_i+\widetilde\phi^-_i$, where
\begin{align*}
\widetilde\phi^+_i(x,v)&=e^{-\frac{1}{\tau |v_1|}\int_0^x \tilde\nu_i(y)dy }f_{i,L}(v)+\frac{1}{\tau |v_1|}\int_0^xe^{-\frac{1}{\tau |v_1|}\int_y^x\tilde\nu_i(z)dz}\tilde\nu_i\widetilde{\mathcal{M}}_idy,\\
\widetilde\phi^-_i(x,v)&=e^{-\frac{1}{\tau |v_1|}\int_x^1 \tilde\nu_i(y)dy }f_{i,R}(v)+\frac{1}{\tau |v_1|}\int_x^1e^{-\frac{1}{\tau |v_1|}\int_x^y\tilde\nu_i(z)dz}\tilde\nu_i\widetilde{\mathcal{M}}_idy.
\end{align*}

To apply the Banach fixed point theorem and conclude our main results, we need to prove that $\Phi$ ($\widetilde\Phi$) maps $\Omega$ into $\Omega$, and $\Phi$ ($\widetilde\Phi$) is a contraction on $\Omega$, under the assumption of Theorem \ref{main}  (Theorem \ref{main2}). These are proved respectively in Proposition \ref{gamma in gamma} in Section 6 and Proposition \ref{Contraction} in Section 7.
\section{$\Phi$ maps $\Omega$ into itself}
 The main goal of this section is stated in the following proposition. Since the arguments are similar, we provide detail mainly for the solution operator $\Phi$ for the slow reaction model \ref{system}.  
\begin{proposition}\label{gamma in gamma} (1) Assume the assumptions in Threom \ref{main} are satisfied. Let $f\in\Omega$. Then, $\Phi(f)\in\Omega$ for sufficiently large $\tau$.\newline
\noindent (2) Assume the assumptions in Threom \ref{main2} are satisfied. Let $f\in\Omega$. Then, $\widetilde\Phi(f)\in\Omega$ for sufficiently large $\tau$.

	\end{proposition}
\begin{remark} We only consider the slow reaction model. The proof for the fast reaction model is identical. 
\end{remark}
	\begin{proof}
The proof is divided into Lemma \ref{prop1}, Lemma \ref{prop2}, Lemma \ref{prop3}, and Lemma \ref{prop4}
below.
\end{proof}.
\begin{lemma}\label{estimate4} Let $f\in\Omega$. Then there exist positive constants $C_{l,u}$ such that
\begin{align*}
\begin{split}
\mathcal{M}_i(1+|v|^2)\leq C_{l,u}\exp{\bigg(-C_{l,u}|v|^2\bigg)}.
\end{split}
\end{align*}
\begin{proof} Lemma \ref{ub} and \ref{lb} imply that
\begin{align*}
\mathcal{M}_i&=n_i\bigg(\frac{m_i}{2k\pi T_i}\bigg)^{3/2} \exp{\bigg(-\frac{m_i|v-U_i|^2}{2kT_i}\bigg)}\\
&\leq C_{l,u}\exp{\bigg(\frac{-m_i|v-U_i|^2}{2kT_i}\bigg)}\\
&\leq C_{l,u}\exp{\bigg(\frac{m_i|U_i|^2}{2kT_i}\bigg)}\exp{\bigg(\frac{-m_i|v|^2}{4kT_i}\bigg)}\\
&\leq C_{l,u}\exp{\bigg(-C_{l,u}|v|^2\bigg)}.
\end{align*}
And for $|v|^2\mathcal{M}_1$, we know
\begin{align*}
|v|^2\mathcal{M}_i &\leq C_{l,u}\exp{\bigg(-C_{l,u}|v|^2\bigg)}|v|^2
\leq C_{l,u} \exp{\bigg(-C_{l,u}|v|^2\bigg)},
\end{align*}
where we use $x^2e^{-x^2}<C$ for some $C>0$.
\end{proof}
\end{lemma}
\begin{lemma}\label{prop1} Let $f\in\Omega$. Assume $f_{i,L}$ and $f_{i,R}$ satisfy all assumptions in Theorem \ref{main}. Then
\begin{align*}
\phi_i\geq0.
\end{align*}
\begin{proof} By Lemma \ref{ub}, we have
	\begin{align*}
	\mathcal{M}_i&=n_i\bigg(\frac{m_i}{2\pi k T_i}\bigg)^{3/2} \exp{\bigg(-\frac{m_i|v-U_i|^2}{2 kT_i}\bigg)}\\&\geq a_{i,l}\bigg(\frac{m_i}{2\pi k T_u}\bigg)^{3/2}\exp{\bigg(-\frac{m_i|v-U_i|^2}{2 kT_i}\bigg)}\\&>0.
	\end{align*}
	Hence,
\begin{align*}
\phi_i&\geq e^{-\frac{1}{\tau |v_1|}\int_0^x \nu_i(y)dy }f_{i,L}(v)1_{v_1>0} + e^{-\frac{1}{\tau |v_1|}\int_x^1 \nu_i(y)dy }f_{i,R}(v)1_{v_1<0}\geq 0.
\end{align*}
\end{proof}
\end{lemma}
\begin{lemma}\label{l} Assume $f_{i,L}$ and $f_{i,R}$ satisfy all assumptions in Theorem \ref{main}. Then, for sufficiently large $\tau$, we have
	\begin{align*}
	\int_{v_1>0}e^{-\frac{1}{\tau |v_1|}\int_0^x \nu_i(y)dy}f_{i,L}(v)
	\begin{pmatrix}
	1\\
	|v_1|\\
	|v|^2
	\end{pmatrix} dv
	\geq 
	\frac{1}{4}\int_{v_1>0}f_{i,L}(v)
	\begin{pmatrix}
	1\\
	|v_1|\\
	|v|^2
	\end{pmatrix} dv
	\end{align*}
	and
	\begin{align*}
	\int_{v_1<0}e^{-\frac{1}{\tau |v_1|}\int_x^1 \nu_i(y)dy}f_{i,R}(v)
	\begin{pmatrix}
	1\\
	|v_1|\\
	|v|^2
	\end{pmatrix} dv
	\geq 
	\frac{1}{4}\int_{v_1<0}f_{i,R}(v)
	\begin{pmatrix}
	1\\
	|v_1|\\
	|v|^2
	\end{pmatrix} dv.
	\end{align*}
	\begin{proof} Take $r>0$ small enough so that
		\begin{align*}
		\int_{v_1\geq r}f_{i,L}(v)dv\geq\frac{1}{2}\int_{v_1>0}f_{i,L}(v)dv.
		\end{align*}
		Then for sufficiently large $\tau$,
		\begin{align*}
		\int_{v_1>0}e^{-\frac{1}{\tau |v_1|}\int_0^x \nu_i(y)dy}f_{i,R}(v)dv&\geq e^{-\frac{C_{l,u}}{\tau r}}\int_{v_1>r}f_{i,L}dv
		\geq \frac{1}{4}\int_{v_1>0}f_{i,L}dv
		=a_{i,l}.
		\end{align*}
	Other estimates can be proved by the same argument. We omit it.
	\end{proof}
\end{lemma}
\begin{lemma}\label{prop2} Assume $f\in\Omega$ and 
	$f_{i,L}$ and $f_{i,R}$ satisfy the assumptions in Theorem \ref{main}. Then we have
\begin{align*}
a_{i,l}\leq\int_{\mathbb{R}^3}\phi_i\ dv, \quad c_{i,l}\leq\int_{\mathbb{R}^3}|v|^2\phi_i\ dv.
\end{align*}
\begin{proof} We know
\begin{align*}
\phi_i&\geq e^{-\frac{1}{\tau |v_1|}\int_0^x \nu_i(y)dy}f_{i,L}(v)1_{v_1>0}+e^{-\frac{1}{\tau|v_1|}\int_x^1\nu_i(y)dy}f_{i,R}(v)1_{v_1>0}.
\end{align*}
Intergrating with respect to $dv$ and $|v|^2dv$, we obtain from Lemma \ref{l} that
\begin{align*}
\int_{\mathbb{R}^3}\phi_idv\geq a_{i,l}
\end{align*}
and
\begin{align*}
\int_{\mathbb{R}^3}|v|^2\phi_idv \geq c_{1,l}.
\end{align*}
\end{proof}
\end{lemma}
\begin{lemma}\label{prop3} Let $f \in \Omega$. Assume $f_{i,L}$ and $f_{i,R}$ satisfy the assumptions in Theorem \ref{main}. Then for  $\tau>0$ sufficiently large,  we have
\begin{align*}
\begin{split}
\int_{\mathbb{R}^3}\phi_i dv\leq a_{i,u}, \quad \int_{\mathbb{R}^3}|v|^2\phi_i dv\leq c_{i,u}.
\end{split}
\end{align*}
\begin{proof} We define $I$ and $II$ from
\begin{align*}
\int_{\mathbb{R}^3}\phi^+_idv=&\int_{v_1>0}e^{-\frac{1}{\tau|v_1|}\int_0^x \nu_i(y)dy}f_{i,L}(v)dv\\
&+\int_{v_1>0}\int_0^x \frac{1}{\tau|v_1|}e^{-\frac{1}{\tau|v_1|}\int_y^x \nu_i(z)dz}\nu_i\mathcal{M}_idydv\\
=&I+II.
\end{align*}
For $I$, we have,
\begin{align}\label{I}
\int_{v_1>0}e^{-\frac{1}{\tau|v_1|}\int_0^x \nu_i(y)dy}f_{i,L}(v)dv\leq\int_{v_1>0}f_{i,L}(v)dv.
\end{align}
By Lemma \ref{estimate4}, we compute $II$ as
\begin{align}\label{computation}
\begin{split}
\int_{v_1>0}&\int_0^x \frac{1}{\tau|v_1|}e^{-\frac{1}{\tau|v_1|}\int_y^x \nu_i(z)dz}\nu_i\mathcal{M}_idv\\
\leq& C_{l,u}\int_{v_1>0}\int_0^x \frac{1}{\tau|v_1|}e^{-\sum_{r=1}^{4}\nu_{ir}a_l/\tau|v_1|(x-y)}e^{-C_{l,u}|v|^2}dydv\\
\leq& C_{l,u}\bigg(\int_0^x\int_{v_1>0}\frac{1}{\tau|v_1|}e^{-\sum_{r=1}^{4}\nu_{ir}a_l/\tau|v_1|(x-y)}e^{-C_{l,u}|v_1|^2}dv_1dy\bigg)\\
&\times\bigg(\int_{\mathbb{R}^2}e^{-C_{l,u}(|v_2|^2+|v_3|^2)}dv_2dv_3\bigg)\\
\leq& C_{l,u}\int_0^x\int_{v_1>0}\frac{1}{\tau|v_1|}e^{-\sum_{r=1}^{4}\nu_{ir}a_l/\tau|v_1|(x-y)}e^{-C_{l,u}|v_1|^2}dv_1dy\\
=:&C_{l,u}\overline{II}.
\end{split}
\end{align}
We divide the domain of integration as follows:
\begin{align*}
\overline{II}&=\bigg\{\int_0^x\int_{0<v_1<\frac{1}{\tau}}+\int_0^x\int_{\frac{1}{\tau}<v_1<\tau}+\int_0^x\int_{\tau<v_1}\bigg\}\frac{1}{\tau|v_1|}e^{-\sum_{r=1}^{4}\nu_{ir}a_l/\tau|v_1|(x-y)}e^{-C_{l,u}|v_1|^2}dv_1dy\\
&=:A+B+C.
\end{align*}
For $A$, we compute
\begin{align*}
A&=\int_{0<v_1<\frac{1}{\tau}}\int_0^x \frac{1}{\tau|v_1|}e^{-\sum_{r=1}^{4}\nu_{ir}a_l/\tau|v_1|(x-y)}e^{-C_{l,u}|v_1|^2}dydv_1\\
&=\frac{1}{\sum_{r=1}^{4}\nu_{ir}a_l}\int_{0<v_1<\frac{1}{\tau}}\bigg( 1-e^{-\sum_{r=1}^{4}\nu_{ir}a_l/\tau|v_1|x}\bigg)e^{-C_{l,u}|v_1|^2}dv_1\\
&\leq \frac{1}{\sum_{r=1}^{4}\nu_{ir}a_l}\int_{0<v_1<\frac{1}{\tau}}1\ dv_1\\
&\leq \frac{1}{\sum_{r=1}^{4}\nu_{ir}a_l},
\end{align*}
where we used $1-e^{-\frac{\nu_1}{|v_1|}}\leq1$ and $e^{-C_{l,u}|v_1|^2}\leq1$. Similarly we estimate $B$ as
\begin{align*}
B&\leq\frac{1}{\sum_{r=1}^{4}\nu_{ir}a_l}\int_{\frac{1}{\tau}<v_1<\tau}\bigg( 1-e^{-\sum_{r=1}^{4}\nu_{ir}a_l/\tau|v_1|x}\bigg)e^{-C_{l,u}|v_1|^2}dv_1\\
&\leq \int_{\frac{1}{\tau}<v_1<\tau} \frac{1}{\tau|v_1|}dv_1\\
&= \frac{2}{\tau}\ln{\tau},
\end{align*}
where we used $1-e^{-x}\leq x$.\\
Finally, we compute 
\begin{align*}
C&\leq \int_{\tau<v_1}\int_0^x \frac{1}{\tau|v_1|}e^{-\sum_{r=1}^{4}\nu_{ir}a_l/\tau|v_1|(x-y)}e^{-C_{l,u}|v_1|^2}dydv_1\\
&\leq \frac{1}{\tau^2} \int_{\mathbb{R}}e^{-C_{l,u}|v_1|^2}dv_1\\
&\leq C_{l,u}\frac{1}{\tau^2}.
\end{align*}
Summarizing the estimates for $A,B$ and $C$, we obtain 
\begin{align}\label{II}
II\leq C_{l,u}\bigg\{\frac{1}{\tau}+\frac{\ln{\tau}}{\tau}+\frac{1}{\tau^2}\bigg\}\leq C_{l,u}\bigg\{\frac{\ln{\tau}+1}{\tau}\bigg\}.
\end{align}
Combining (\ref{I}) with (\ref{II}), we have
\begin{align*}
\int_{\mathbb{R}^3}\phi^+_idv \leq\int_{v_1>0}f_{I,L}(v)dv+ C_{l,u}\bigg\{\frac{\ln{\tau}+1}{\tau}\bigg\}.
\end{align*}
We can derive similar estimate for $\phi^-_i$:
\begin{align*}
\int_{\mathbb{R}^3}\phi^-_idv \leq\int_{v_1<0}f_{I,R}(v)dv+ C_{l,u}\bigg\{\frac{\ln{\tau}+1}{\tau}\bigg\},
\end{align*}
and hence
\begin{align*}
\int_{\mathbb{R}^3}\phi_i dv\leq \frac{a_u}{2}+C_{l,u}\bigg\{\frac{\ln{\tau}+1}{\tau}\bigg\}.
\end{align*}
By choosing sufficiently large $\tau>0$, we get the desired result. The proof for the second estimate 
is almost identical. We omit it.
\end{proof}
\end{lemma}
\begin{lemma}\label{estimate7} Let $f\in\Omega$. 
	Assume $f_{i,L}$ and $f_{i,R}$ satisfy the assumptions in Theorem \ref{main}.
	Then for $j=2,3$, we have
\begin{align*}
\bigg|\int_{\mathbb{R}^3}\phi_iv_jdv\bigg|\leq C_{l,u}\bigg(\frac{\ln{\tau}+1}{\tau}\bigg).
\end{align*}
\begin{proof} We consider this only for $\phi_i^+$ because the other case can be proved by similar ways. Integrating $\phi^+_i$ with respect to $v_2dv_2dv_3$, we have
\begin{align*}
\int_{\mathbb{R}^2}\phi^+_iv_2dv_2dv_3 =&\ e^{-\frac{1}{\tau|v_1|}\int_0^x \nu_i(y)dy}\int_{\mathbb{R}^2}f_{i,L}(v)v_2dv_2dv_3\\
& + \frac{1}{\tau|v_1|}\int_0^x e^{-\frac{1}{\tau|v_1|}\int_y^x \nu_i(z)dz} \nu_i(y)\bigg(\int_{\mathbb{R}^2}\mathcal{M}_iv_2dv_2dv_3\bigg)dy.
\end{align*}
By our assumption on $f_{i,L}$, it can be reduced to the following:
\begin{align}\label{computation2}
\int_{\mathbb{R}^2}\phi^+_iv_2dv_2dv_3=\frac{1}{\tau|v_1|}\int_0^x e^{-\frac{1}{\tau|v_1|}\int_y^x\nu_i(z)dz}\nu_i(y) \bigg(\int_{\mathbb{R}^2}\mathcal{M}_iv_2dv_2dv_3\bigg)dy.
\end{align}
As in the computation in (\ref{computation}), we see
\begin{align*}
\int_{\mathbb{R}^2}\mathcal{M}_iv_2dv_2dv_3 &\leq C_{l,u}e^{-C_{l,u}|v_1|^2}\int_{\mathbb{R}^2}e^{-C_{l,u}(|v_2|^2+|v_3|^2)}|v_2|dv_2dv_3\\
&\leq C_{l,u}e^{-C_{l,u}|v_1|^2}.
\end{align*}
Substituting this in (\ref{computation2}) and then integrating on $v_1>0$, we get
\begin{align*}
\int_{\mathbb{R}^3}\phi^+_iv_2dv &\leq C_{l,u}\int_0^x\int_{v_1>0}\frac{1}{\tau|v_1|}e^{-\frac{1}{\tau|v_1|}\int_y^x\nu_i(z)dz}e^{-C_{l,u}|v_1|^2}dv_1dy\\
&\leq C_{l,u}\bigg\{\frac{\ln{\tau}+1}{\tau}\bigg\}
\end{align*}
where we had the last inequality from (\ref{computation}) and (\ref{II}).
\end{proof}
\end{lemma}
\begin{lemma}\label{prop4}
Let $f\in\Omega$. Assume $f_{i,L}$ and $f_{i,R}$ satisfy the assumptions in Theorem \ref{main}. Then, for sufficiently large $\tau>0$, we have
\begin{align*}
\bigg(\int_{\mathbb{R}^3}\phi_idv\bigg)\bigg(\int_{\mathbb{R}^3}\phi_i|v|^2dv\bigg)-\bigg|\int_{\mathbb{R}^3}\phi_ivdv\bigg|^2\geq\gamma_l.
\end{align*}
\begin{proof}
Applying the Cauchy-Schwarz inequality, we have
\begin{align*}
&\bigg(\int_{\mathbb{R}^3}\phi_idv\bigg)\bigg(\int_{\mathbb{R}^3}\phi_i|v|^2dv\bigg)-\bigg|\int_{\mathbb{R}^3}\phi_ivdv\bigg|^2\\
&\qquad\geq \bigg(\int_{\mathbb{R}^3}\phi_i|v|dv\bigg)^2-\bigg|\int_{\mathbb{R}^3}\phi_ivdv\bigg|^2\\
&\qquad\geq \bigg(\int_{\mathbb{R}^3}\phi_i|v_1|dv\bigg)^2-\bigg|\int_{\mathbb{R}^3}\phi_ivdv\bigg|^2.
\end{align*}
And we decompose the last term as
\begin{align*}
&\bigg(\int_{\mathbb{R}^3}\phi_i|v_1|dv\bigg)^2-\bigg|\int_{\mathbb{R}^3}\phi_ivdv\bigg|^2\\
&=\bigg(\int_{\mathbb{R}^3}\phi_i|v_1|dv\bigg)^2-\bigg[\bigg(\int_{\mathbb{R}^3}\phi_iv_1dv\bigg)^2+\underbracket{\bigg(\int_{\mathbb{R}^3}\phi_iv_2dv\bigg)^2+\bigg(\int_{\mathbb{R}^3}\phi_iv_3dv\bigg)^2}_{=:R}\bigg].\\
\end{align*}
In view of Lemma \ref{estimate7}, we have
\begin{align*}
R\leq C_{l,u}\bigg(\frac{\ln{\tau}+1}{\tau}\bigg).
\end{align*}
On the other hand, since
\begin{align*}
&\bigg(\int_{\mathbb{R}^3}\phi_i|v_1|dv\bigg)^2-\bigg(\int_{\mathbb{R}^3}\phi_iv_1dv\bigg)^2\\
&\qquad\geq \bigg(\int_{\mathbb{R}^3}\phi_i(|v_1|+v_1)dv\bigg)\bigg(\int_{\mathbb{R}^3}\phi_i(|v_1|-v_1)dv\bigg)\\
&\qquad=4 \bigg(\int_{v_1>0}\phi_i|v_1|dv\bigg)\bigg(\int_{v_1<0}\phi_i|v_1|dv\bigg),
\end{align*}
Lemma \ref{l} implies that
\begin{align*}
&\bigg(\int_{\mathbb{R}^3}\phi_i|v_1|dv\bigg)^2-\bigg(\int_{\mathbb{R}^3}\phi_iv_1dv\bigg)^2\\
&\qquad\geq 4\bigg(\int_{v_1>0}e^{-\frac{1}{\tau|v_1|}\int_0^x\nu_i(y)dy}f_{i,L}(v)|v_1|dv\bigg)\bigg(\int_{v_1<0}e^{-\frac{1}{\tau|v_1|}\int_x^1\nu_i(y)dy}f_{i,R}(v)|v_1|dv\bigg)\\
&\qquad\geq\frac{1}{4}\bigg(\int_{v_1>0}f_{i,L}(v)|v_1|dv\bigg)\bigg(\int_{v_1<0}f_{i,R}(v)|v_1|dv\bigg)\\
&\qquad= 4\gamma_l.
\end{align*}
In conclusion, for sufficiently large $\tau>0$, we obtain
\begin{align*}
&\bigg(\int_{\mathbb{R}^3}\phi_idv\bigg)\bigg(\int_{\mathbb{R}^3}\phi_i|v|^2dv\bigg)-\bigg|\int_{\mathbb{R}^3}\phi_ivdv\bigg|^2 \\
&\qquad\geq 4\gamma_l-C_{l,u}\bigg(\frac{\ln{\tau}+1}{\tau}\bigg)\\
&\qquad\geq \gamma_l.
\end{align*}
\end{proof}
\end{lemma}
%
%
%
%
\section{$\Phi$ is contractive in $\Omega$} It remains  to show the solution map $\Phi$ and $\widetilde\Phi$ are contraction maps. We start with the esitmates for the single component macroscopic fields and global macroscopic fields, which holds commonly for the first and second model. 
\begin{lemma}\label{lemma45} Let $f=(f_1,f_2,f_3,f_4),\ g= (g_1,g_2,g_3,g_4)\in\Omega$. Then we have:\newline
	\noindent(1) The single component macroscopic parameters satisfy
\begin{align*}
|n_f^{(i)}-n_g^{(i)}|,~ |U_f^{(i}-U_g^{(i)}|, ~|T_f^{(i)}-T_g^{(i)}|\leq C_{l,u}\sup\limits_{x\in[0,1]}||f_i-g_i||_{L_2^1}.
\end{align*}
\noindent (2) The global macroscopic parameters satisfy
\begin{align*}
|n_f-n_g|,|U_f-U_g|,|T_f-T_g|\leq C_{l,u}d(f,g).
\end{align*}
\begin{proof}
\noindent (1) The first estimate is straightforward:
\begin{align*}
|n_f^{(i)}-n_g^{(i)}|=\int_{\mathbb{R}^3}|f_i-g_i|dv\leq \sup\limits_{x\in[0,1]}||f_i-g_i||_{L_2^1}.
\end{align*}
For the second estimate, we use $\rho_f^{(i)}\geq m_ia_{i,l}$ to get
\begin{align*}
|U_f^{(i)}-U_g^{(i)}|&\leq\frac{1}{\rho_f^{(i)}}|\rho_f^{(i)}U_f^{(i)}-\rho_g^{(i)}U_g^{(i)}|+\frac{1}{\rho_f^{(i)}}|\rho_f^{(i)}-\rho_g^{(i)}|\ |U_g^{(i)}|\\
&\leq \frac{m_i}{\rho_f^{(i)}}\int_{\mathbb{R}^3}|f_i-g_i|\ |v|dv+\frac{m_i|U_g^{(i)}|}{\rho_f^{(i)}}\int_{\mathbb{R}^3}|f_i-g_i|dv\\
&\leq C_{l,u}\sup\limits_{x\in[0,1]}||f_i-g_i||_{L_2^1}.
\end{align*}
For the third estimate, we decompose 
\begin{align*}
|T_f^{(i)}-T_g^{(i)}|&\leq \frac{1}{n_f^{(i)}}|n_f^{(i)}T_f^{(i)}-n_g^{(i)}T_g^{(i)}|+\frac{1}{n_f^{(i)}}|n_f^{(i)}-n_g^{(i)}|\ |T_g^{(i)}|\\
&\leq \frac{m_i}{3kn_f^{(i)}}\int_{\mathbb{R}^3}\bigg|f_i|v-U_f^{(i)}|^2-g_i|v-U_g^{(i)}|^2\bigg|dv +\frac{m_i|T_g^{(i)}|}{n_f^{(i)}}\int_{\mathbb{R}^3}|f_i-g_i|dv\\
&=I+II.
\end{align*}
Then, $n_f^{(i)}\geq a_{i,l}$ and Lemma \ref{scmp}  gives 
$$II\leq C_{l,u}\sup\limits_{x\in[0,1]}||f_i-g_i||_{L_2^1},$$
and
\begin{align*}
I&\leq\frac{m_i}{ka_{i,l}}\int_{\mathbb{R}^3}\bigg|f_i|v-U_f^{(i)}|^2-g_i|v-U_g^{(i)}|^2\bigg|dv\\
&\leq \frac{m_i}{ka_{i,l}}\int_{\mathbb{R}^3}\bigg|(f_i-g_i)|v-U_f^{(i)}|^2+g_i(|v-U_f^{(i)}|^2-|v-U_g^{(i)}|^2)\bigg|dv\\
&=\frac{m_i}{ka_{i,l}}\int_{\mathbb{R}^3}\bigg|(f_i-g_i)|v-U_f^{(i)}|^2+g_i(2v-U_f^{(i)}-U_g^{(i)})(U_f^{(i)}-U_g^{(i)})\bigg|dv\\
&\leq C_{l,u}\int_{\mathbb{R}^3}|f_i-g_i|(1+|v|^2)+|g_i|(1+|v|)|U_f^{(i)}-U_g^{(i)}|dv\\
&\leq C_{l,u}\sup\limits_{x\in[0,1]}||f_i-g_i||_{L_2^1}.
\end{align*}
\noindent (2) The estimates for the global macroscopic parameters follows directly from (1). We omit the proof.
\end{proof}
\end{lemma}
And then, we consider the estimates for the reactive parameters and the collision frequency for the first model (\ref{system}).
\begin{lemma}\label{lemma455} Let $f=(f_1,f_2,f_3,f_4),\ g= (g_1,g_2,g_3,g_4)\in\Omega$. Then we have
	\begin{align*}
	|n_{f,i}-n_{g,i}|,|U_{f,i}-U_{g,i}|,|T_{f,i}-T_{g,i}|,|\nu_{f,i}-\nu_{g,i}|\leq C_{l,u}\sum_{j=1}^{4}\sup\limits_{x\in[0,1]}||f_j-g_j||_{L_2^1}.
	\end{align*}
\end{lemma}
	\begin{proof}
		We first establish the following claim:
			\begin{align}\label{S}
		|\mathcal{S}_f-\mathcal{S}_g|\leq C_{l,u}d(f,g),
		\end{align}
		and
			\begin{align}\label{SS}
		|\nu_{f,i}-\nu_{g,i}|\leq C_{l,u}\sum_{j=1}^{4}\sup\limits_{x\in[0,1]}||f_j-g_j||_{L_2^1}.
		\end{align}
		For simplicity, we donote $K(x)=\Gamma(3/2,x)e^x$  and compute
		\begin{align}\label{S2}
		\begin{split}
		|\mathcal{S}_f-\mathcal{S}_g|\leq& C_{l,u}\bigg[\ \bigg|K\bigg(\frac{\Delta E}{kT_f}\bigg)-K\bigg(\frac{\Delta E}{kT_g}\bigg)\bigg|n_f^{(3)}n_f^{(4)}+K\bigg(\frac{\Delta E}{kT_g}\bigg)|n_f^{(3)}-n_g^{(3)}|n_f^{(4)}\\
		&+K\bigg(\frac{\Delta E}{kT_g}\bigg)n_g^{(3)}|n_f^{(4)}-n_g^{(4)}|\bigg]\\
		&+C_{l,u}\bigg[\ \bigg|\Gamma\bigg(\frac{3}{2},\frac{\Delta E}{kT_f}\bigg)-\Gamma\bigg(\frac{3}{2},\frac{\Delta E}{kT_g}\bigg)\bigg|n_f^{(3)}n_f^{(4)}+\Gamma\bigg(\frac{3}{2},\frac{\Delta E}{kT_g}\bigg)|n_f^{(3)}-n_g^{(3)}|n_g^{(4)}\\
		&+\Gamma\bigg(\frac{3}{2},\frac{\Delta E}{kT_g}\bigg)\bigg)n_g^{(3)}|n_f^{(4)}-n_g^{(4)}|\bigg].
		\end{split}
		\end{align}
		Since $f, g\in\Omega$, $T_f$ and $T_g$ are bounded from below and above by constants defined in terms of constants given in (\ref{parameters1}), (\ref{parameters2}) and (\ref{parameters3}) (Lemma \ref{scmp}). Therefore, since $K$ and $\Gamma$ are continuously differentiable, we derive
		\begin{align*}
		K\bigg(\frac{\Delta E}{kT_g}\bigg),\ \Gamma\bigg(\frac{3}{2},\frac{\Delta E}{kT_g}\bigg)\leq C_{l,u}
		\end{align*}
		and
		\begin{align*}
		&\bigg|K\bigg(\frac{\Delta E}{kT_f}\bigg)-K\bigg(\frac{\Delta E}{kT_g}\bigg)\bigg|,\  \bigg|\Gamma\bigg(\frac{3}{2},\frac{\Delta E}{kT_f}\bigg)-\Gamma\bigg(\frac{3}{2},\frac{\Delta E}{kT_g}\bigg)\bigg|\\
		&\qquad\leq C_{l,u}\bigg| \frac{\Delta E}{kT_f}-\frac{\Delta E}{kT_g} \bigg|\quad \text{by the Mean value theorem}\\
		&\qquad\leq C_{l,u}\bigg| \frac{\Delta E(T_g-T_f)}{kT_fT_g}\bigg|\\
		&\qquad\leq C_{l,u}|T_f-T_g|
		\end{align*}
		Combining these with (\ref{S2}) and Lemma \ref{lemma45} proves the first estimate of the claim. Once (\ref{S}) is established, the second estimate of the claim follows directly from the definition of $\nu_i$ and Lemma \ref{lemma45}.
		
		From (\ref{S}) and (\ref{SS}), we have
		\begin{align*}
		|n_{f,i}-n_{g,i}|\leq& |n_f^{(i)}-n_g^{(i)}|+|\mathcal{S}_f|\bigg|\frac{\nu_{f,i}-\nu_{g,i}}{\nu_{f,i}\nu_{g,i}}\bigg|+\frac{1}{\nu_{g,i}}|\mathcal{S}_f-\mathcal{S}_g|\\
		\leq& C_{l,u}\sum_{j=1}^{4}\sup\limits_{x\in[0,1]}||f_j-g_j||_{L_2^1}.
		\end{align*}
	We recall the definition of $U_i$
		\begin{align*}
		U_i&=\frac{n^{(i)}}{n_i}U^{(i)}+\frac{2}{m_i\nu_i}\sum_{j=1}^{4}\chi_{ij}\mu_{ij}\frac{n^{(i)}}{n_i}n^{(j)}(U^{(j)}-U^{(i)})+\frac{\lambda_i}{n_i\nu_i}U\mathcal{S},
		\end{align*}
	and use (\ref{S}), (\ref{SS}) and Lemma \ref{scmp} to get
		\begin{align*}
		|U_{f,i}-U_{g,i}|\leq C_{l,u}\sum_{j=1}^{4}\sup\limits_{x\in[0,1]}||f_j-g_j||_{L_2^1}.
		\end{align*}
 The proof for the remaining estimates are almost identical, we omit it.
		\end{proof}
	And then, we consider the estimates for the reactive parameters and the collision frequency for the first model (\ref{system}).
\begin{lemma} Let $f=(f_1,f_2,f_3,f_4),\ g= (g_1,g_2,g_3,g_4)\in\Omega$. Then we have
	\begin{align*}
	|\tilde n_{f,i}-\tilde n_{g,i}|,|\tilde U_{f}-\tilde U_{g}|,|\tilde T_{f}-\tilde T_{g}|,|\tilde\nu_{f,i}-\tilde\nu_{g,i}|\leq C_{l,u}\sum_{j=1}^{4}\sup\limits_{x\in[0,1]}||f_j-g_j||_{L_2^1}.
	\end{align*}
\end{lemma}
	\begin{proof}
	We recall from the proof of (4.8) that $\tilde n_{f,1}=G(\bold n_f,\bold n_f,\boldsymbol{\tilde\nu_f},\boldsymbol{\tilde\nu_f},\bold T_f,\bold V_f)$ and $\tilde n_{g,1}=G(\bold n_g,\bold n_g,\boldsymbol{\tilde\nu_g},\boldsymbol{\tilde\nu_g},\bold T_g,\bold V_g)$, and for each $(\bold x, \bold y, \boldsymbol\mu,\boldsymbol\eta,\boldsymbol\alpha,\boldsymbol\beta)\in (\mathbb{R}_+)^{20}\times\mathbb{R}^4$ we have
		\begin{align*}
		\frac{d}{dz}&F_{\bold x, \bold y, \boldsymbol\mu,\boldsymbol\eta,\boldsymbol\alpha,\boldsymbol\beta}(z)
		=  \frac{1}{z}+\frac{\mu_1}{\mu_2x_2+\mu_1z-\eta_1y_1}+\frac{\mu_1}{\eta_3y_3-\mu_1z+\eta_1y_1}\\
		&+\frac{\mu_1}{\eta_4y_4-\mu_1z+\eta_1y_1}+\frac{\frac{3}{2}(\Delta E)^2\mu_1\sum\limits_{i=1}^4 \eta_i y_i}{\bigg[\sum\limits_{i=1}^4\mu_i x_i\Big(\frac{1}{2}m_i(\beta_i^2)+\frac{3}{2}k\alpha_i\Big)+\Delta E(\mu_1z-\eta_1y_1)\bigg]^2}>0.
		\end{align*}
By Implicit function theorem, $G$ is continuously differentiable at all $(\bold x, \bold y, \boldsymbol\mu,\boldsymbol\eta,\boldsymbol\alpha,\boldsymbol\beta)\in(\mathbb{R}_+)^{20}\times\mathbb{R}^4$. Hence the gradient of $G$ is bounded on the following compact set:
\[
K:=[a_l,a_u]^8\times[\nu_m,\nu_M]^8\times [T_l,T_u]^4\times [-R,R]^4.
\] 
Therefore, by the mean value theorem, we get 
\begin{align*}
|\tilde n_{f,1}-\tilde n_{g,1}|&\leq \max_{(\bold x, \bold y, \boldsymbol\mu,\boldsymbol\eta,\boldsymbol\alpha,\boldsymbol\beta)\in K}|\nabla G(\bold x, \bold y, \boldsymbol\mu,\boldsymbol\eta,\boldsymbol\alpha,\boldsymbol\beta)|\\
&\times|(\bold n_f,\bold n_f,\boldsymbol{\tilde\nu_f},\boldsymbol{\tilde\nu_f},\bold T_f,\bold V_f)-(\bold n_g,\bold n_g,\boldsymbol{\tilde\nu_g},\boldsymbol{\tilde\nu_g},\bold T_g,\bold V_g)|.
\end{align*}
 These, together with Lemma \ref{lemma45} imply $|\tilde n_{f,1}-\tilde n_{f,2}|\leq C_{l,u}\sum_{j=1}^{4}\sup\limits_{x\in[0,1]}||f_j-g_j||_{L_2^1}.$
 The proof for the other estimates are, albeit more tedious, essentially same. We omit it.
\end{proof}

\begin{lemma}\label{estimate8} Let $f=(f_1,f_2,f_3,f_4)\in\Omega$ and $g\in(g_1,g_2,g_3,g_4)\in\Omega$. Then, we have
\begin{align*}
|\mathcal{M}(f_i)-\mathcal{M}(g_i)|\leq C_{l,u}d(f,g),
\end{align*}
and
\begin{align*}
|\widetilde{\mathcal{M}}(f_i)-\widetilde{\mathcal{M}}(g_i)|\leq C_{l,u}d(f,g).
\end{align*}
\begin{proof} Since the proof is identical, we only consider the first estimate.  We denote $\mathcal{M}(f_i):=\mathcal{M}(m_i,n_{f,i},U_{f,i},T_{f,i})$, $\mathcal{M}(g_i):=\mathcal{M}(m_i,n_{g,i},U_{g,i},T_{g,i})$
and apply  Taylor expansion to write $\mathcal{M}(f_i)-\mathcal{M}(g_i)$ as
\begin{align*}
\mathcal{M}(f_i)-\mathcal{M}(g_i)=&(n_{f,i}-n_{g,i})\int_0^1\frac{\partial \mathcal{M}(\theta)}{\partial n}d\theta\\
&+(U_{f,i}-U_{g,i})\int_0^1\frac{\partial \mathcal{M}(\theta)}{\partial U}d\theta\\
&+(T_{f,i}-T_{g,i})\int_0^1\frac{\partial \mathcal{M}(\theta)}{\partial T}d\theta\\
&=:A+B+C,
\end{align*}
where
\begin{align*}
\frac{\partial \mathcal{M}(\theta)}{\partial X} = \frac{\partial \mathcal{M}(\theta)}{\partial X}(m_i,n_\theta,U_\theta,T_\theta)
\end{align*}
for $(n_\theta,U_\theta,T_\theta)=(1-\theta)(n_{f,i},U_{f,i},T_{f,i})+\theta(n_{g,i},U_{g,i},T_{g,i}).$
For $A$, we observe
\begin{align*}
\frac{\partial \mathcal{M}(\theta)}{\partial n}=\frac{1}{n_\theta}\mathcal{M}(\theta),
\end{align*}
so that
\begin{align*}
\bigg|\frac{\partial \mathcal{M}(\theta)}{\partial\rho}\bigg|\leq C_{l,u}e^{-C_{l,u}|v|^2},
\end{align*}
from Lemma \ref{estimate4}.
\\

For $B$, we similarly  observe
\begin{align*}
\frac{\partial \mathcal{M}(\theta)}{\partial U}=\frac{m_i(v-U_\theta)}{kT_\theta}\mathcal{M}(\theta),
\end{align*}
which implies
\begin{align*}
\bigg|\frac{\partial \mathcal{M}(\theta)}{\partial U}\bigg|&\leq C_{l,u}(1+|v|)\mathcal{M}(\theta)\\
&\leq C_{l,u}e^{-C_{l,u}|v|^2},
\end{align*}
by Lemma \ref{estimate4} and Lemma \ref{ub}.
\\

Finally, we compute the derivative w.r.t $T$ as 
\begin{align*}
\frac{\partial \mathcal{M}(\theta)}{\partial T}=\bigg\{-\frac{3}{2T_\theta}+\frac{m_i|v-U_\theta|^2}{2kT_\theta^2}\bigg\}\mathcal{M}(\theta),
\end{align*}
and apply Lemma \ref{estimate4}, Lemma \ref{ub} and Lemma \ref{lb} to get
\begin{align*}
\bigg|\frac{\partial \mathcal{M}(\theta)}{\partial T}\bigg|\leq C_{l,u}(1+|v|^2)e^{-C_{l,u}|v|^2}\leq
C_{l,u}e^{-C_{l,u}|v|^2}.
\end{align*}
Combining all these estimates, we obtain
\begin{align*}
|\mathcal{M}(f_i)-\mathcal{M}(g_i)|\leq C_{l,u}\bigg\{|n_{f,i}-n_{g,i}|+|U_{f,i}-U_{g,i}|+|T_{f,i}-T_{g,i}|\bigg\}e^{-C_{l,u}|v|^2}.
\end{align*}
This, together with Lemma \ref{lemma45}, gives the desired result.
\end{proof}
\end{lemma}
\begin{proposition} \label{Contraction} Assume $f_{i,L}$ and $f_{i,R}$ satisfy the assumptions in Theorem \ref{main} or \ref{main2}. Let $f=(f_1,f_2,f_3,f_4)$, $g\in(g_1,g_2,g_3, g_4) \in\Omega$. Then there exists a $\alpha\in(0,1)$ such that
\begin{align*}
d(\Phi(f),\Phi(g))\leq \alpha d(f,g).
\end{align*}
and 
\begin{align*}
d(\widetilde\Phi(f),\widetilde\Phi(g))\leq \alpha d(f,g).
\end{align*}
if $\tau$ is taken sufficiently large.
\begin{proof}
The proof is almost identical for both case. We only consider the first estimate. 
Also, we only compute $|\phi^+(f_i)-\phi^+(g_i)|$ because the argument for  $|\phi^-(f_i)-\phi^-(g_i)|$ is same.\\
Consider
\begin{align*}
\phi^+(f_i)-\phi^+(g_i)&=\bigg\{e^{-\frac{1}{\tau|v_1|}\int_0^x\nu_{f,i}(y)dy}-e^{-\frac{1}{\tau|v_1|}\int_0^x\nu_{g,i}(y)dy}\bigg\}f_{1,L}(v)\\
&+\frac{1}{\tau |v_1|}\bigg(\int_0^xe^{-\frac{1}{\tau |v_1|}\int_y^x\nu_{f,i}(z)dz}\nu_{f,i}(y)\mathcal{M}(f_i)dy\\
&-\int_0^xe^{-\frac{1}{\tau |v_1|}\int_y^x\nu_{g,i}(z)dz}\nu_{g,i}(y)\mathcal{M}(g_1)dy\bigg)\\
&=I+II.
\end{align*}
By the mean value theorem, there exists $0<\theta<1$ such that
\begin{align}\label{I2}
\begin{split}
I=&\bigg\{e^{-\frac{1}{\tau|v_1|}\int_0^x\nu_{f,i}(y)dy}-e^{-\frac{1}{\tau|v_1|}\int_0^x\nu_{f,i}(y)dy}\bigg\}f_{i,L}(v)\\
&=-\frac{1}{\tau|v_1|}e^{-\frac{1}{\tau|v_1|}\int_0^x (1-\theta)\nu_{f,i}(y)+\theta\nu_{g,i}(y)dy}\left(\int_0^x \left\{\nu_{f,i}(y)-\nu_{g,i}(y)\right\} dy\right) f_{i,L}(v).
\end{split}
\end{align}
Therefore, by $n_{f}^{(i)},n_{g}^{(i)}\geq a_l$, and Lemma \ref{lemma45}, we obtain
\begin{align}\label{I3}
\begin{split}
|I|&\leq \frac{1}{\tau |v_1|}\bigg(e^{-\frac{1}{\tau|v_1|}\int_0^x\sum_{j=1}^{4}\nu_{ij}a_ldy}\int_0^x |\nu_{f,i}-\nu_{g,i}|dy\bigg)f_{1,L}(v)\\
&\leq \frac{C_{l,u}}{\tau|v_1|}e^{-\frac{\sum_{j=1}^{4}\nu_{ij}a_l}{\tau|v_1|}}f_{1,L}d(f,g)\\
&\leq \frac{C_{l,u}}{\tau|v_1|}f_{1,L}d(f,g).
\end{split}
\end{align}
We divide the estimate of  $II$ into the following three parts. First, by a similiar way as in the proof for $I$, we estimate the difference of integrating factor as
\begin{align}\label{II2}
\begin{split}
&\frac{1}{\tau |v_1|}\bigg|\int_0^xe^{-\frac{1}{\tau |v_1|}\int_y^x\nu_{f,i}(z)dz}\nu_{f,i}(y)\mathcal{M}(f_i)dy-\int_0^xe^{-\frac{1}{\tau |v_1|}\int_y^x\nu_{g,i}(z)dz}\nu_{f,i}(y)\mathcal{M}(f_i)dy\bigg|\\
&\leq \frac{1}{\tau |v_1|}\int_0^x \frac{1}{\tau|v_1|}e^{-\frac{1}{\tau|v_1|}\int_y^x (1-\theta)\nu_{f,i}(z)+\theta\nu_{g,i}(z)dz}\int_y^x |\nu_{f,i}(z)-\nu_{g,i}(z)| dz\ \nu_{f,i}(y)\mathcal{M}(f_i)dy\\
&\leq \left(\frac{C_{l,u}}{\tau |v_1|}\int_0^x \frac{1}{\tau|v_1|}e^{-\frac{\sum_{j=1}^{4}\nu_{ij}a_l}{\tau|v_1|}}\mathcal{M}(f_i)dy\right) d(f,g)\\
&\leq \left(\frac{C_{l,u}}{\tau |v_1|}\int_0^x e^{-\frac{\sum_{j=1}^{4}\nu_{ij}a_l}{2\tau|v_1|}}\mathcal{M}(f_i)dy\right) d(f,g),
\end{split}
\end{align}
where we used that $xe^{-x}<C$ for some $C>0$. Secondly we use (\ref{SS}) to estimate the difference
of the collision frequency:
\begin{align}\label{II3}
\begin{split}
&\frac{1}{\tau |v_1|}\int_0^xe^{-\frac{1}{\tau |v_1|}\int_y^x\nu_{g,i}(z)dz}|\nu_{f,i}(y)-\nu_{g,i}(y)|\mathcal{M}(f_1)dy\\
&\leq \frac{C_{l,u}}{\tau|v_1|}\int_0^x e^{-\frac{\sum_{j=1}^{4}\nu_{ij}a_l}{\tau|v_1|}(x-y)}\mathcal{M}(f_1)dy\cdot d(f,g).
\end{split}
\end{align}\\
Finally, by Lemma \ref{estimate8}), we estimate the difference of the Maxwellians:
\begin{align}\label{II4}
\begin{split}
&\frac{1}{\tau |v_1|}\int_0^xe^{-\frac{1}{\tau |v_1|}\int_y^x\nu_{g,i}(z)dz}\nu_{g,i}(z)\left\{\mathcal{M}(f_i)-\mathcal{M}(g_i)\right\}dy\\
&\leq \frac{C_{l,u}}{\tau|v_1|}\int_0^x e^{-\frac{\sum_{j=1}^{4}\nu_{ij}a_l}{\tau|v_1|}(x-y)}e^{-C_{l,u}|v|^2}dy\cdot d(f,g).\\
\end{split}
\end{align}\\
Combining (\ref{I3}), (\ref{II2}), (\ref{II3}), and (\ref{II4}), we obtain
\begin{align*}
|\phi^+(f_i)-\phi^+(g_i)|
\leq\ & C_{l,u}\cdot d(f,g)\cdot \bigg(\frac{1}{\tau|v_1|}f_{i,L}+\frac{1}{\tau|v_1|}\int_0^x e^{-\frac{\sum_{j=1}^{4}\nu_{ij}a_l}{\tau|v_1|}(x-y)}\mathcal{M}(f_1)dy\\ 
&+ \frac{1}{\tau|v_1|}\int_0^x e^{-\frac{\sum_{j=1}^{4}\nu_{ij}a_l}{\tau|v_1|}(x-y)}e^{-C_{l,u}|v|^2}dy\bigg).
\end{align*}\\
Therefore,
\begin{align*}
||\phi^+(f_i)-\phi^+(g_i)||_{L_2^1}&\leq C_{l,u}\cdot d(f,g) \cdot \bigg( \int_{v_1>0}\int_0^x \frac{1}{\tau|v_1|}f_{i,L}(1+|v|^2)dydv\\
&+\int_{v_1>0}\int_0^x \frac{1}{\tau|v_1|}e^{-\frac{\sum_{j=1}^{4}\nu_{ij}a_l}{\tau|v_1|}(x-y)}\mathcal{M}(f_i)(1+|v|^2)dydv\\
&+\int_{v_1>0}\int_0^x \frac{1}{\tau|v_1|}e^{-\frac{\sum_{j=1}^{4}\nu_{ij}a_l}{\tau|v_1|}(x-y)}e^{-C_{l,u}|v|^2}(1+|v|^2)dydv\bigg).
\end{align*}
Applying Lemma \ref{estimate4}, we have
\begin{align*}
||\phi^+(f_i)-\phi^+(g_i)||_{L_2^1}&\leq C_{l,u}\cdot d(f,g) \cdot \bigg( \int_{v_1>0}\int_0^x \frac{1}{\tau|v_1|}f_{i,L}(1+|v|^2)dydv\\
&+\int_{v_1>0}\int_0^x \frac{1}{\tau|v_1|}e^{-\frac{\sum_{j=1}^{4}\nu_{ij}a_l}{\tau|v_1|}(x-y)}e^{-C_{l,u}|v|^2}dydv\bigg),
\end{align*}
Then, from the same computation as in the esitmate of $\overline{II}$ in (\ref{computation}), we obtain
\begin{align*}
||\phi^+(f_i)-\phi^+(g_i)||_{L_2^1}&\leq C_{l,u}\bigg[\frac{a_{i,s}+c_{i,s}}{\tau}+\bigg(\frac{\ln{\tau}+1}{\tau}\bigg)\bigg]d(f,g)
\leq C_{l,u}\bigg(\frac{\ln{\tau}+1}{\tau}\bigg)d(f,g).
\end{align*}
By a simliar argument, we have
\begin{align*}
||\phi^-(f_i)-\phi^-(g_i)||_{L_2^1}\leq C_{l,u}\bigg(\frac{\ln{\tau}+1}{\tau}\bigg)d(f,g).
\end{align*}
Hence,
\begin{align*}
||\phi(f_i)-\phi(g_i)||_{L_2^1}\leq C_{l,u}\bigg(\frac{\ln{\tau}+1}{\tau}\bigg)d(f,g).
\end{align*}
Taking supremum on both sides and choosing sufficiently large $\tau>0$, we get the desired result.
\end{proof}
\end{proposition}

{\bf Acknowledgement}
Doheon Kim was supported by a KIAS Individual Grant (MG073901) at Korea Institute for Advanced Study. The work of Seok-Bae Yun was supported by Samsung Science and Technology Foundation under
Project Number SSTF-BA1801-02.
\appendix


\begin{thebibliography}{99}
	\bibitem{AAP}
	Andries, P., Aoki, K.,  and Perthame, B.: A consistent BGK-type model for
	gas mixtures. J. Stat. Phys. 106, 993 (2002).

 
 \bibitem{ACI} 
 Arkeryd, L., Cercignani, C., Illner, R.:  Measure solutions of the steady Boltzmann equation in a slab. Comm. Math. Phys. {\bf142} (1991), no. 2, 285-–296.
 
 \bibitem{AN1} 
 Arkeryd, L.; Nouri, A. A compactness result related to the stationary Boltzmann equation in a slab, with applications to the existence theory. Indiana Univ. Math. J. {\bf44} (1995), no. 3, 815-–839.
 
 \bibitem{AN2} 
 Arkeryd, L., Nouri, A.:  $L^1$ solutions to the stationary Boltzmann equation in a slab. Ann. Fac. Sci. Toulouse Math. (6) {\bf9} (2000), no. 3, 375–-413.

 \bibitem{AN3} 
 Arkeryd, L., Nouri, A.: The stationary Boltzmann equation in the slab with given weighted mass for hard and soft forces. Ann. Scuola Norm. Sup. Pisa Cl. Sci. (4) {\bf27} (1998), no. 3-4, 533–-556 (1999).
 
 \bibitem{BGCY} Bae, G.-C., Yun, S.-B.: Stationary Quantum BGK model for Bosons and Fermions in a bounded interval. Submitted.
  
  \bibitem{Bang Y} 
  Bang, J., Yun, S.-B.: Stationary solutions for the ellipsoidal BGK model in a slab. J. Differential Equations {\bf261} (2016), no. 10, 5803-–5828. 
 
   \bibitem{BGK} Bhatnagar, P. L., Gross, E. P. and Krook, M.: A model for collision processes in gases. Small amplitude process in charged
 and neutral one-component systems, Physical Revies, {\bf 94} (1954), 511-525.

 \bibitem{BCMR}
  Bisi, M., Conforto, F., Monaco, R., Ricciardello, A.: On the steady deflagration process for a gas mixture undergoing irreversible reactions. Ric. Mat. {\bf68} (2019), no. 1, 13-–35.
 
 \bibitem{BisiS}
 Bisi, M., Spiga, G.: On kinetic models for polyatomic gases and their hydrodynamic limits. Ric. Mat. {\bf66} (2017), no. 1, 113–-124.
 
 \bibitem{BS} 
 Bisi, M., Spiga, G. : On a kinetic BGK model for slow chemical reactions. Kinet. Relat. Models {\bf4} (2011), no. 1, 153-–167.
 

 


 
 \bibitem{B1} 
 Brull, S.: The stationary Boltzmann equation for a two-component gas for soft forces in the slab. Math. Methods Appl. Sci. {bf31} (2008), no. 14, 1653–1666. 35F30 (76P05 82C40)
 
 \bibitem{B2} 
 Brull, S.:  The stationary Boltzmann equation for a two-component gas in the slab with different molecular masses. Adv. Differential Equations {\bf15} (2010), no. 11-12, 1103–1124.
 
 \bibitem{B3} 
 Brull, S., Schneider, J.: A new approach for the ellipsoidal statistical model. Contin. Mech. Thermodyn. {\bf20} (2008), no.2, 63-74.
 
 \bibitem{BrullS}
 Brull, S., Schneider, J.:
 Derivation of a BGK model for reacting gas mixtures. Commun. Math. Sci. {\bf12} (2014), no. 7, 1199–-1223.
 
 
 \bibitem{ELM}
 Esposito, R., Lebowitz, J.L., Marra, R.: Hydrodynamic limit of the stationary Boltzmann equation in
 a slab. Commun. Math. Phys. {\bf160}, 49-–80 (1994)
 
 \bibitem{EGKM} 
 Esposito, R., Guo, Y., Kim, C., Marra, R.: Stationary solutions to the Boltzmann equation in the hydrodynamic limit. Ann. PDE {\bf4} (2018), no. 1, Art. 1, 119 pp. 
 
 \bibitem{Giu1}
 Guiraud, J.P.: Probleme aux limites intérieur pour l’équation de Boltzmann linéaire. J. Méc. {\bf9}, 183–231
 (1970)
 
 \bibitem{Giu2}
 Guiraud, J.P.: Probleme aux limites intérieur pour l’équation de Boltzmann en régime stationnaire,
 faiblement non linéaire. J. Méc. {11}, 443-–490 (1972)
 
\bibitem{Ghomeshi}
Ghomeshi, S.: 
Existence and uniqueness of solutions for the Couette problem. 
J. Stat. Phys. {\bf118} (2005), no. 1-2, 265–-300.

	
	
	\bibitem{GRS}
	M. Groppi, S. Rjasanow and G. Spiga: “A kinetic relaxation approach to fast reactive mixtures: shock wave structure." J. Stat. Mech.-Theory Exp. (2009), P10010
	
	\bibitem{GS}
	M. Groppi and G. Spiga. A Bhatnagar–Gross–Krook-type approach for chemically reacting
	gas mixtures. Phys.
	Fluids, {\bf16} (12):4273-4284, (2004).
	
	
	\bibitem{GS} 
	Groppi, M., Spiga, G. :
	A kinetic relaxation model for bimolecular chemical reactions. Bull. Inst. Math. Acad. Sin. (N.S.) {\bf2} (2007), no. 2, 609-–635.
	
	\bibitem{GST}
	 Groppi, M., Spiga, G., Takata, S.: The steady shock problem in reactive gas mixtures. Bull. Inst. Math. Acad. Sin. (N.S.) {\bf2} (2007), no. 4, 935–-956.
	
	\bibitem{GRS} 
	Groppi, M.  Rjasanow, S. and Spiga, G.: A kinetic relaxation approach to fast reactive mixtures:
	Shock wave structure, J. Stat. Mech. Theory. Exp., P10010, 2009.
	
	\bibitem{HY} Hwang, B.-H., Yun, S.-B.: Stationary solutions to the boundary value problem for the relativistic BGK model in a slab. Kinet. Relat. Models {\bf12} (2019), no. 4, 749–-764. 
	
	
	\bibitem{Maslova}
	Maslova, N. B.: Nonlinear evolution equations. Kinetic approach. Series on Advances in Mathematics for Applied Sciences, {\bf10}. World Scientific Publishing Co., Inc., River Edge, NJ, 1993. x+193 pp.
	
	\bibitem{MB} Monaco, R.,  and M. Pandolfi B.: A BGK-type model for a gas mixture
	with reversible reactions. New Trends in Mathematical Physics (World Scientific, Singapore, 2004).
	
	
	\bibitem{Nouri}
	Nouri, A.: An existence result for a quantum BGK model. Math. Comput. Modelling {\bf47} (2008), no. 3-4, 515–-529
	

	

	
	

	\bibitem{Ukai} Ukai, S.: Stationary solutions of the BGK model equation on a finite interval with large boundary data. Transport theory Statist. Phys.{\bf 21} (1992) no.4-6.


	\bibitem{RS} Rossani, A., Spiga, G.: A note on the kinetic theory of chemically reacting gases, Physica A,
	{\bf272}, 563–-573, 1999.
	


	
\end{thebibliography}
\end{document}